\documentclass{article}

\usepackage{amscd,amsmath,amssymb,amsthm}
\usepackage{color}

\newcommand{\eps}{\varepsilon}

\newcommand{\R}{\mathbb R}
\newcommand{\N}{\mathbb N}

\newcommand{\then}{\Longrightarrow}
\newcommand{\J}{{\cal J}}

\DeclareMathOperator*{\esssup}{ess\; sup}
\DeclareMathOperator*{\essinf}{ess\; inf}

\newcommand\meas{{\rm meas}}
\newcommand\e{{\rm e}}

\newtheorem{corollary}{Corollary}[section]
\newtheorem{theorem}[corollary]{Theorem}
\newtheorem{lemma}[corollary]{Lemma}
\newtheorem{proposition}[corollary]{Proposition}

\theoremstyle{definition}
\newtheorem{definition}[corollary]{Definition}
\newtheorem{remark}[corollary]{Remark}

\numberwithin{equation}{section}
\begin{document}

\title{{\bf Existence of radial bounded solutions\\
 for some quasilinear elliptic equations in $\R^N$} 
\footnote{The research that led to the present paper was partially supported 
by MIUR--PRIN project 2017JPCAPN ``Qualitative and quantitative aspects of nonlinear PDEs'',
{\sl Fondi di Ricerca di Ateneo} 2015/16 ``Problemi differenziali non lineari''  
and Research Funds INdAM -- GNAMPA Project
2018 ``Problemi ellittici semilineari: alcune idee variazionali''}}

\author{Anna Maria Candela and Addolorata Salvatore\\
{\small Dipartimento di Matematica}\\
{\small Universit\`a degli Studi di Bari Aldo Moro} \\
{\small Via E. Orabona 4, 70125 Bari, Italy}\\
{\small \it annamaria.candela@uniba.it}\\
{\small \it addolorata.salvatore@uniba.it}}
\date{}

\maketitle

\begin{abstract}
We study the quasilinear equation 
\[(P)\qquad
- {\rm div} (A(x,u) |\nabla u|^{p-2} \nabla u) + \frac1p\ A_t(x,u)
 |\nabla u|^p + |u|^{p-2}u\  =\  g(x,u) \qquad \hbox{in $\R^N$,}
\]
with $N\ge 3$, $p > 1$,  where $A(x,t)$, $A_t(x,t) = \frac{\partial A}{\partial t}(x,t)$ and
$g(x,t)$ are Carath\'eodory functions on $\R^N \times \R$.

Suitable assumptions on $A(x,t)$ and $g(x,t)$ 
set off the variational structure of $(P)$ and its related functional $\J$
is $C^1$ on the Banach space $X = W^{1,p}(\R^N) \cap L^\infty(\R^N)$.
In order to overcome the lack of compactness, we assume
that the problem has radial symmetry, then we look for 
critical points of $\J$ restricted to $X_r$, subspace of the radial functions in $X$.

Following an approach which exploits the interaction between
$\|\cdot\|_X$ and the norm on $W^{1,p}(\R^N)$, we prove the existence 
of at least one weak bounded radial solution of $(P)$ by applying a generalized version
of the Ambrosetti--Rabinowitz Mountain Pass Theorem.
\end{abstract}



\section{Introduction} \label{secintroduction}

In this paper we investigate the existence of weak bounded radial solutions of the quasilinear equation 
\begin{equation}\label{euler}
- {\rm div} (A(x,u) |\nabla u|^{p-2} \nabla u) + \frac1p\ A_t(x,u)
 |\nabla u|^p + |u|^{p-2}u\  =\  g(x,u) \qquad \hbox{in $\R^N$,}
\end{equation}
with $N\ge 3$, $p > 1$,  where $A(x,t)$, $g(x,t)$ are
given real functions on $\R^N \times \R$ and $A_t(x,t) = \frac{\partial}{\partial t} A(x,t)$.

Equation \eqref{euler}, with $p =2$, is
related to the research of standing waves for the ``modified Schr\"odinger
equations'' and appears quite naturally in Mathematical Physics,
derived as model of several physical phenomena  
in plasma physics, fluidmechanics, theory of
Heisenberg ferromagnets and magnons, dissipative quantum
mechanics and condensed matter theory (for more details, see, e.g., \cite{LWW} and references therein).

In the mathematical literature,
very few results are known about equation \eqref{euler} if $A_t(x,t) \not\equiv 0$ 
since, in general, a classical variational approach fails.  
In fact, the ``natural'' functional associated to \eqref{euler} is
\[
\J(u)\ =\  \frac1p\ \int_{\R^N} A(x,u)|\nabla u|^p dx + \frac1p\ \int_{\R^N} |u|^p dx - \int_{\R^N} G(x,u) dx,
\]
which is not defined in $W^{1,p}(\R^N)$ for a general coefficient $A(x,t)$
in the principal part.  
Moreover, even if $A(x,t)$ is a smooth strictly positive bounded function,
so the functional $\J$ is well defined in $W^{1,p}(\R^N)$,
if $A_t(x,t) \not\equiv 0$ it is G\^ateaux differentiable only along directions of 
$W^{1,p}(\R^N) \cap L^\infty(\R^N)$. 

In the past, such a problem has been overcome by introducing suitable definitions
of critical point for non--differentiable functionals (see,
e.g., \cite{AB1,CD,CDM}). In particular, existence results have been obtained if equation \eqref{euler}
is given on a bounded domain with homogeneous Dirichlet boundary conditions
(see, e.g., \cite{AB1,Ca} and also \cite{BP} and references therein),
while some existence results in unbounded domain have been stated, e.g.,  
in \cite{AG}.

In the whole Euclidean space $\R^N$ other existence results have been proved 
by means of constrained minimization arguments (see \cite{LW,PSW})
or by using a suitable change of variable (see, e.g., \cite{CJ,LWW}). We note that this last method 
works only if $A(x,t)$ has a very special form, in particular it is independent of $x$. 

More recently, if $\Omega$ is a bounded subset of $\R^N$,
a different approach has been developed which
exploits the interaction between two different norms on $W^{1,p}_0(\Omega) \cap L^\infty(\Omega)$
(see \cite{CP1,CP2,CPS1}).

Following this way of thinking, here firstly we prove that, under some quite
natural conditions, functional $\J$ is $C^1$ in the Banach space $X = W^{1,p}(\R^N) \cap L^\infty(\R^N)$
equipped with the intersection norm $\|\cdot\|_X$ (see Proposition \ref{smooth1}).

Then, due to the lack of compactness of our setting,
we assume that the problem has radial symmetry 
so we study the existence of critical points of $\J$ 
restricted to the subspace $X_r$ of the radial functions.

So, by using the interaction between the norm $\|\cdot\|_X$   
and the standard one on $W^{1,p}(\R^N)$, if $G(x,t)$ has a subcritical growth, we state  
that $\J$ satisfies a weaker version of the Cerami's variant of the Palais--Smale condition 
in $X_r$ (see Definition \ref{wCPSdef} and Proposition \ref{wCPS}). 
We note that, in general, $\J$ cannot verify the standard Palais--Smale condition, or its Cerami's
variant, as Palais--Smale sequences may converge in 
the $W^{1,p}(\R^N)$--norm but be unbounded in $L^{\infty}(\R^N)$
(see, e.g., \cite[Example 4.3]{CP2017}). 

Since our main theorem requires a list of hypotheses, we give 
its complete statement in Section \ref{secmain} (see Theorem \ref{mainthm}),
while here, in order to highlight at least a model problem and its 
related result, we consider the particular setting of the ``modified Schr\"odinger equation'' in $\R^3$
with $p=2$ and
\[
A(x,t) \ =\ A_1(x) + A_2(x) |t|^{2 s}, \qquad 
g(x,t)\ =\ |t|^{\mu-2}t,
\]
so that problem \eqref{euler} reduces to
\begin{equation}\label{euler0}
- {\rm div} ((A_1(x) + A_2(x) |u|^{2 s}) \nabla u) + s A_2(x)|u|^{2 s-2} u |\nabla u|^2
+ u\ =\ |u|^{\mu-2}u\qquad\hbox{in $\R^3$,}
\end{equation}
thus generalizing the particular case $A_1(x) \equiv A_2(x) \equiv 1$
and $s=1$ which many papers deal with (see, e.g., \cite{CJ}). Here, we state the following result.

\begin{theorem}\label{mod0}
Let $A_1$, $A_2 \in L^{\infty}(\Omega)$ be two radially symmetric functions such that 
\[
A_1(x) \ge \alpha_0,\qquad A_2(x) \ge 0\qquad \hbox{a.e. in $\R^3$,}
\]
for a constant $\alpha_0 > 0$. If $3 < 2(1+s) < \mu< 6$,
then problem \eqref{euler0} has at least one weak bounded radial solution.
\end{theorem}

We note that the same hypotheses of Theorem \ref{mod0} appear in \cite{BP},
where it is stated the existence of a bounded positive solution of equation \eqref{euler0} in a 
bounded domain. However, in such a paper, by using
a different approach, namely a sequence of truncated functionals,
the authors are able also to study the case $0< 2s \le 1$.
\smallskip

This paper is organized as follows.
In Section \ref{secabstract} we introduce the weak Cerami--Palais--Smale condition
and a generalized version of the Mountain Pass Theorem in an abstract setting.
On the contrary, in Section \ref{variational}, 
we give the first hypotheses on functions $A(x,t)$, $g(x,t)$, 
and the variational formulation of our problem when no radial assumption is involved. 
Then, in Section \ref{secmain} the main result is stated 
and the radial symmetric setting is pointed out.
At last, in Section \ref{secproof} the main theorem is proved.


\section{Abstract tools} \label{secabstract}

In this section, we denote
by $(X, \|\cdot\|_X)$ a Banach space with dual space $(X',\|\cdot\|_{X'})$,
by $(W,\|\cdot\|_W)$ another Banach space such that
$X \hookrightarrow W$ continuously,
and by $J: X \to \R$ a given $C^1$ functional.

Taking $\beta \in \R$, we say that a sequence
$(u_n)_n\subset X$ is a {\sl Cerami--Palais--Smale sequence at level $\beta$},
briefly {\sl $(CPS)_\beta$--sequence}, if
\[
\lim_{n \to +\infty}J(u_n) = \beta\quad\mbox{and}\quad 
\lim_{n \to +\infty}\|dJ(u_n)\|_{X'} (1 + \|u_n\|_X) = 0.
\]
Moreover, $\beta$ is a {\sl Cerami--Palais--Smale level}, briefly {\sl $(CPS)$--level}, 
if there exists a $(CPS)_\beta$--sequence.

The functional $J$ satisfies the classical Cerami--Palais--Smale condition in $X$ 
at level $\beta$ if every $(CPS)_\beta$--sequence converges in $X$
up to subsequences. Anyway, thinking about the setting of our problem,
in general $(CPS)_\beta$--sequences may also exist which are unbounded in $\|\cdot\|_X$
but converge with respect to $\|\cdot\|_W$. Then, we can weaken the classical Cerami--Palais--Smale 
condition in the following way.  

\begin{definition} \label{wCPSdef}
The functional $J$ satisfies the
{\slshape weak Cerami--Palais--Smale 
condition at level $\beta$} ($\beta \in \R$), 
briefly {\sl $(wCPS)_\beta$ condition}, if for every $(CPS)_\beta$--sequence $(u_n)_n$,
a point $u \in X$ exists such that 
\begin{description}{}{}
\item[{\sl (i)}] $\displaystyle 
\lim_{n \to+\infty} \|u_n - u\|_W = 0\quad$ (up to subsequences),
\item[{\sl (ii)}] $J(u) = \beta$, $\; dJ(u) = 0$.
\end{description}
If $J$ satisfies the $(wCPS)_\beta$ condition at each level $\beta \in I$, $I$ real interval, 
we say that $J$ satisfies the $(wCPS)$ condition in $I$.
\end{definition} 

Due to the convergence only in the $W$--norm, the $(wCPS)_\beta$
condition implies that the set of critical points of $J$ at level $\beta$
is compact with respect to $\|\cdot\|_W$; anyway,
this weaker ``compactness'' assumption is enough to prove 
a Deformation Lemma and then some 
abstract critical point theorems (see \cite{CP3}).
In particular, the following generalization of the Mountain Pass Theorem
\cite[Theorem 2.1]{AR} can be stated.

\begin{theorem}[Mountain Pass Theorem]
\label{mountainpass}
Let $J\in C^1(X,\R)$ be such that $J(0) = 0$
and the $(wCPS)$ condition holds in $\R$.
Moreover, assume that two constants
$r$, $\varrho > 0$ and a point $e \in X$ exist such that
\[
u \in X, \; \|u\|_W = r\quad \then\quad J(u) \ge \varrho,
\]
\[
\|e\|_W > r\qquad\hbox{and}\qquad J(e) < \varrho.
\]
Then, $J$ has a Mountain Pass critical point $u^* \in X$ such that $J(u^*) \ge \varrho$.
\end{theorem}

\begin{proof}
For the proof, see \cite[Theorem 1.7]{CP3} as  
the required assumption $(wC)$, namely any $(CPS)$--level is also a critical level, 
follows from the stronger $(wCPS)$ condition.
\end{proof}


\section{Variational setting and first properties}
\label{variational}

Here and in the following, $|\cdot|$ is the standard norm 
on any Euclidean space as the dimension
of the considered vector is clear and no ambiguity arises.
Furthermore, we denote by:
\begin{itemize}
\item $B_1(0) = \{x \in \R^N: |x| < 1\}$ the open unit ball in $\R^N$; 
\item $\meas(\Omega)$ the usual Lebesgue measure of a measurable set $\Omega$ in $\R^N$;
\item $L^l(\R^N)$ the Lebesgue space with
norm $|u|_l = \left(\int_{\R^N}|u|^l dx\right)^{1/l}$ if $1 \le l < +\infty$;
\item $L^\infty(\R^N)$ the space of Lebesgue--measurable 
and essentially bounded functions $u :\R^N \to \R$ with norm
\[
|u|_{\infty} = \esssup_{\R^N} |u|;
\]
\item $W^{1,p}(\R^N)$ the classical Sobolev space with
norm $\|u\|_{W} = (|\nabla u|_p^p + |u|_p^p)^{\frac1p}$ if $1 \le p < +\infty$;
\item $W_r^{1,p}(\R^N) = \{u \in W^{1,p}(\R^N): u(x) = u(|x|)\}$ 
the subspace of $W^{1,p}(\R^N)$ equipped with 
the same norm $\|\cdot\|_W$ with dual space $(W_r^{1,p}(\R^N))'$.
\end{itemize}

From the Sobolev Imbedding Theorems, for any $l \in [p,p^*]$
with $p^* = \frac{pN}{N-p}$ if $N > p$, or any $l \in [p,+\infty[$ if $p = N$,
the Sobolev space $W^{1,p}(\R^N)$ is continuously imbedded in $L^l(\R^N)$, i.e.,
a constant $\sigma_l > 0$ exists, such that 
\begin{equation}\label{Sob1}
|u|_l\ \le\ \sigma_l \|u\|_W \quad \hbox{for all $u \in W^{1,p}(\R^N)$}
\end{equation}
(see, e.g., \cite[Corollaries 9.10 and 9.11]{Br}). Clearly, it is $\sigma_p = 1$.
On the other hand, if $p > N$ then  $W^{1,p}(\R^N)$ is continuously imbedded in $L^\infty(\R^N)$
(see, e.g., \cite[Theorem 9.12]{Br}).

Thus, we define 
\begin{equation}\label{space}
X := W^{1,p}(\R^N) \cap L^\infty(\R^N),\qquad
\|u\|_X = \|u\|_W + |u|_\infty.
\end{equation}

From now on, we assume $p \le N$ as, otherwise, it is $X = W^{1,p}(\R^N)$ and
the proofs can be simplified.

\begin{lemma}\label{immergo}
For any $l \ge p$ the Banach space $X$ is continuously imbedded in $L^l(\R^N)$, i.e.,
a constant $\sigma_l > 0$ exists such that 
\begin{equation}\label{Sob2}
|u|_l\ \le\ \sigma_l \|u\|_X \quad \hbox{for all $u \in X$.}
\end{equation}
\end{lemma}

\begin{proof}
If $p = N$ or if $1\le p<N$ and $l\le p^*$ inequality \eqref{Sob2}
follows from \eqref{Sob1} and \eqref{space}. \\
On the other hand, if $1\le p<N$ and $l > p^*$ then, taking any $u \in X$,
again \eqref{space} implies
\[
\int_{\R^N} |u|^l dx\ \le\ |u|_\infty^{l-p} \int_{\R^N} |u|^{p} dx \le 
|u|_\infty^{l-p} \|u\|_W^{p} \le \|u\|_X^{l},
\] 
thus \eqref{Sob2} holds with $\sigma_l =1$.
\end{proof}

From Lemma \ref{immergo} it follows that if
$(u_n)_n \subset X$, $u \in X$ are such that
$u_n \to u$ in $X$, then $u_n \to u$ also in $L^l(\R^N)$
for any $l \ge p$. This result can be weakened as follows. 

\begin{lemma}\label{immergo2}
If $(u_n)_n \subset X$, $u \in X$, $M > 0$ are such that
\begin{equation}\label{succ1}
\|u_n - u\|_W \to 0 \ \quad\hbox{if $n \to+\infty$,}
\end{equation}
\begin{equation}\label{succ2}
|u_n|_\infty \le M\quad \hbox{for all $n \in \N$,}
\end{equation}
then $u_n \to u$ also in $L^l(\R^N)$ for any $l \ge p$. 
\end{lemma}

\begin{proof}
Let $1 \le p < N$ and $l > p^*$ (otherwise, it is a direct consequence of \eqref{Sob1}). 
Then, from \eqref{Sob1} we have that
\[
\int_{\R^N} |u_n - u|^l dx\ \le\ |u_n - u|_\infty^{l-p} \int_{\R^N} |u_n - u|^{p} dx\ \le\ 
(M + |u|_\infty)^{l-p} \|u_n - u\|_W^{p} 
\] 
which implies the thesis.
\end{proof}

As useful in the following, we recall a technical lemma (see \cite{GM}).

\begin{lemma}\label{tech}
A constant $b_0 > 0$ exists such that for any $y$, $z \in \R^N$, $N \ge1$, it results
\begin{eqnarray}\label{tech1}
||y|^{l-2}y - |z|^{l-2}z| &\le& b_0 |y-z| (|y|+|z|)^{l-2}\qquad\hbox{if $l > 2$,}\\
\label{tech2}
||y|^{l-2}y - |z|^{l-2}z| &\le& b_0 |y-z|^{l-1}\qquad\qquad\qquad\hbox{if $1 < l \le 2$.}
\end{eqnarray}
\end{lemma}

The following estimate is a direct consequence of Lemma \ref{tech}.

\begin{lemma}\label{tech3}
If $p > 1$ a constant $b_1 = b_1(p) > 0$ exists such that 
\[
\int_{\R^N} ||\nabla u|^p - |\nabla v|^p| dx\ \le\ b_1\ \|u-v\|_W 
\big(\|u\|_W^{p-1} + \|v\|_W^{p-1}\big)
\quad\hbox{for any $u$, $v \in W^{1,p}(\R^N)$.} 
\]
\end{lemma}

\begin{proof}
Taking $u$, $v \in W^{1,p}(\R^N)$, by using \eqref{tech1} with $y = |\nabla u(x)|$,
$z = |\nabla v(x)|$ and $l = p+1$, we have that
\[
||\nabla u|^p - |\nabla v|^p| \le b_0 ||\nabla u| - |\nabla v||\ (|\nabla u| + |\nabla v|)^{p-1} \le
b_0 |\nabla u - \nabla v|\ (|\nabla u| + |\nabla v|)^{p-1} 
\]
for a.e. $x \in \R^N$. Hence, from H\"older inequality and direct computations it follows that
\[
\int_{\R^N} ||\nabla u|^p - |\nabla v|^p| dx\ \le\ b_0\ \|u-v\|_W 
\left(2^{p-1} \int_{\R^N}(|\nabla u|^{p} + |\nabla v|^{p})dx\right)^{\frac{p-1}{p}}
\]
which implies the thesis.
\end{proof}

From now on, let $\, A : \R^N \times \R \to \R\,$
and $\, g :\R^N \times \R \to \R\,$ be such that:
\begin{itemize}
\item[$(H_0)$]
$A(x,t)$ is a $C^1$ Carath\'eodory function, i.e., \\
$A(\cdot,t) : x \in \R^N \mapsto A(x,t) \in \R$ is measurable for all $t \in \R$,\\
$A(x,\cdot) : t \in \R \mapsto A(x,t) \in \R$ 
is $C^1$ for a.e. $x \in \R^N$;
\item[$(H_1)$] $A(x,t)$ and $A_t(x,t)$ are essentially bounded
if $t$ is bounded, i.e., 
\[
\sup_{|t| \le r} |A(\cdot,t)| \in L^\infty(\R^N),\quad \sup_{|t| \le r} |A_t(\cdot,t)| \in L^\infty(\R^N)
\qquad \hbox{for any $r > 0$;}
\]
\item[$(G_0)$] 
$g(x,t)$ is a Carath\'eodory function, i.e.,\\
$g(\cdot,t) : x \in \R^N \mapsto g(x,t) \in \R$ is measurable for all $t \in \R$,\\
$g(x,\cdot) : t \in \R \mapsto g(x,t) \in \R$ is continuous for a.e. $x \in \R^N$;
\item[$(G_1)$] $a_1$, $a_2 > 0$ and $q \ge p$ exist such that
\[
|g(x,t)| \le a_1 |t|^{p-1} + a_2 |t|^{q-1} \qquad
\hbox{a.e. in $\R^N$, for all $t \in \R$.}
\]
\end{itemize}
 
\begin{remark}\label{remG}
From $(G_0)$--$(G_1)$ it follows that $\displaystyle G(x,t) = \int_0^t g(x,s) ds$ is
a well defined $C^1$ Ca\-ra\-th\'eo\-do\-ry function in 
$\R^N \times \R$ and
\begin{equation}
\label{alto3}
|G(x,t)| \le \frac{a_1}p |t|^{p} + \frac{a_2}q |t|^{q} \qquad
\hbox{a.e. in $\R^N$, for all $t \in \R$.}
\end{equation}
\end{remark}

We note that \eqref{space} and $(H_1)$ imply
$A(\cdot,u(\cdot))|\nabla u(\cdot)|^p \in L^1(\R^N)$
for all $u \in X$. Furthermore, 
even if no upper bound on $q$ is actually required in $(G_1)$, 
from Lemma \ref{immergo} it is $G(\cdot,u(\cdot)) \in L^1(\R^N)$
for any $u \in X$, too.
Hence, we can consider the functional $\J : X \to \R$ defined as
\begin{equation}
\label{funct}
\J(u)\ =\  \frac1p\ \int_{\R^N} A(x,u)|\nabla u|^p dx + \frac1p\ \int_{\R^N} |u|^p dx - \int_{\R^N} G(x,u) dx,
\qquad
u \in X.
\end{equation}

Taking any $u$, $v\in X$, by direct computations
it follows that its G\^ateaux differential in $u$ along the direction $v$ is
\begin{equation}
\label{diff}
\begin{split}
\langle d\J(u),v\rangle\ =\ &\int_{\R^N} A(x,u) |\nabla u|^{p-2} \nabla u\cdot \nabla v\ dx\ +\ 
\frac1p\ \int_{\R^N} A_t(x,u) v |\nabla u|^{p} dx\\
& + \ \int_{\R^N} |u|^{p-2} u v\ dx
\ -\ \int_{\R^N} g(x,u)v\ dx .
\end{split}
\end{equation}

\begin{proposition}\label{smooth1}
Taking $p > 1$, assume that $(H_0)$--$(H_1)$, $(G_0)$--$(G_1)$ are satisfied.
If $(u_n)_n \subset X$, $u \in X$, $M> 0$ are such that \eqref{succ1}, \eqref{succ2} hold and
\begin{equation}\label{succ3}
 u_n \to u\quad \hbox{a.e. in $\R^N$} \ \quad\hbox{if $n \to+\infty$,}
\end{equation}
then
\[
\J(u_n) \to \J(u)\quad \hbox{and}\quad \|d\J(u_n) - d\J(u)\|_{X'} \to 0
\quad\hbox{if $\ n\to+\infty$.}
\]
Hence, $\J$ is a $C^1$ functional on $X$ with Fr\'echet differential
defined as in \eqref{diff}.
\end{proposition}

\begin{proof}
For simplicity, we set $\J = \J_1 + \J_2$, where
\[
\begin{split}
&\J_1: u \in X \ \mapsto\ \J_1(u) = \frac1p \int_{\R^N} A(x,u) |\nabla u|^p dx \in \R, \\
&\J_2: u \in X \ \mapsto\ \J_2(u) = \frac1p\ \int_{\R^N} |u|^p dx - \int_{\R^N} G(x,u) dx \in \R,
\end{split}
\]
with related G\^ateaux differentials 
\[
\begin{split}
&\langle d\J_1(u),v\rangle\ =\ \int_{\R^N} A(x,u) |\nabla u|^{p-2} \nabla u\cdot \nabla v\ dx\ +\ 
\frac1p\ \int_{\R^N} A_t(x,u) v |\nabla u|^{p} dx\\
&\langle d\J_2(u),v\rangle\ =\  \int_{\R^N} |u|^{p-2}u v \ dx - \int_{\R^N} g(x,u)v\ dx.
\end{split}
\]
Let us consider a sequence $(u_n)_n \subset X$ and $u \in X$, $M>0$ such that \eqref{succ1}, \eqref{succ2} and \eqref{succ3}
hold. Then, \eqref{succ1} implies that 
\begin{equation}\label{bnn}
\|u_n\|_W \le c_0 \quad\hbox{for all $\ n\in \N$}
\end{equation}
for a suitable $c_0 > 0$, while from $(H_1)$ and \eqref{succ2} a constant $c_1 > 0$ exists
such that, for all $n \in \N$, we have
\begin{equation}\label{bn1}
|A(\cdot,u_n(\cdot))|_\infty \le c_1, \quad
|A_t(\cdot,u_n(\cdot))|_\infty \le c_1, \quad
|A(\cdot,u(\cdot))|_\infty \le c_1, \quad |A_t(\cdot,u(\cdot))|_\infty \le c_1.
\end{equation}
Firstly, we prove that
\begin{equation}\label{bnn1}
\J_1(u_n) \to \J_1(u)\quad \hbox{and}\quad \|d\J_1(u_n) - d\J_1(u)\|_{X'} \to 0
\quad\hbox{if $\ n\to+\infty$.}
\end{equation}
To this aim, we note that 
\[
|\J_1(u_n) - \J_1(u)| \le \frac1p \int_{\R^N} |A(x,u_n)|\ ||\nabla u_n|^p - |\nabla u|^p|\ dx +
\frac1p \int_{\R^N} |A(x,u_n) - A(x,u)|\ |\nabla u|^pdx,
\]
where \eqref{bn1} and Lemma \ref{tech3} with \eqref{succ1} and \eqref{bnn} imply that
\[
\int_{\R^N} |A(x,u_n)|\ | |\nabla u_n|^p - |\nabla u|^p|dx \to 0.
\]
Moreover, from $(H_0)$ and \eqref{succ3} it follows that 
\[
|A(x,u_n) - A(x,u)|\ |\nabla u|^p \to 0\qquad \hbox{a.e in $\R^N$,}
\]
while from \eqref{bn1} we obtain that 
\[
|A(x,u_n) - A(x,u)|\ |\nabla u|^p \le 2 c_1 |\nabla u|^p 
\quad \hbox{a.e in $\R^N$, with  $|\nabla u|^p \in L^1(\R^N)$,}
\]
so, from the Lebesgue Dominated Convergence Theorem we have that
\[
\int_{\R^N} |A(x,u_n) - A(x,u)|\ |\nabla u|^p dx \to 0
\]
which implies the first limit in \eqref{bnn1}.\\
Now, taking $v \in X$ such that $\|v\|_X = 1$, we have 
\begin{equation}\label{space1}
|v|_\infty \le 1, \quad \|v\|_W \le 1
\end{equation}
and 
\begin{eqnarray*}
|\langle d\J_1(u_n) - d\J_1(u),v\rangle| &\le&
 \int_{\R^N} |A(x,u_n)|\ ||\nabla u_n|^{p-2}\nabla u_n - |\nabla u|^{p-2}\nabla u||\nabla v|dx\\
&& +\ \int_{\R^N} |A(x,u_n) - A(x,u)| |\nabla u|^{p-1}|\nabla v|dx\\
&& +\ \frac1p\ \int_{\R^N} |A_t(x,u_n)|\ | |\nabla u_n|^p - |\nabla u|^p| |v|dx\\
&& +\ \frac1p\ \int_{\R^N} |A_t(x,u_n) - A_t(x,u)| |\nabla u|^p|v|dx.
\end{eqnarray*}
From  \eqref{bn1}, H\"older inequality and \eqref{space1}, it follows that 
\begin{equation}\label{stima1}
\begin{split}
 &\int_{\R^N} |A(x,u_n)|\ ||\nabla u_n|^{p-2}\nabla u_n - |\nabla u|^{p-2}\nabla u|\ |\nabla v|dx\\
&\qquad \le  c_1 \left( \int_{\R^N} ||\nabla u_n|^{p-2}\nabla u_n - |\nabla u|^{p-2}\nabla u|^{\frac{p}{p-1}}dx\right)^{\frac{p-1}{p}}.
\end{split}
\end{equation}
From one hand, if $p > 2$ from \eqref{tech1}, H\"older inequality with $l = p-1$ and $l' = \frac{p-1}{p-2}$, \eqref{bnn}
and direct computations we have that
\begin{equation}\label{stima2}
\begin{split}
&\left( \int_{\R^N} ||\nabla u_n|^{p-2}\nabla u_n - |\nabla u|^{p-2}\nabla u|^{\frac{p}{p-1}}dx\right)^{\frac{p-1}{p}}\\
&\qquad \le \ b_0 \left( \int_{\R^N} |\nabla u_n - \nabla u|^{\frac{p}{p-1}} 
(|\nabla u_n| + |\nabla u|)^{\frac{p(p-2)}{p-1}}dx\right)^{\frac{p-1}{p}}\\
&\qquad \le \ b_0 \left( \int_{\R^N} |\nabla u_n - \nabla u|^p dx\right)^{\frac{1}{p}} 
\left( \int_{\R^N}(|\nabla u_n| + |\nabla u|)^{p}dx\right)^{\frac{p-2}{p}}\\
&\qquad \le \ c_2 \|u_n - u\|_W
\end{split}
\end{equation} 
for a suitable $c_2 > 0$ independent of $n$. \\
On the other hand, if $1 < p \le 2$ from \eqref{tech2} we have
\begin{equation}\label{stima3}
\left( \int_{\R^N} ||\nabla u_n|^{p-2}\nabla u_n - |\nabla u|^{p-2}\nabla u|^{\frac{p}{p-1}}dx\right)^{\frac{p-1}{p}}\
 \le \ b_0 \|u_n - u\|_W^{p-1}.
\end{equation} 
Whence, from \eqref{stima1}--\eqref{stima3} and \eqref{succ1} it follows that 
\[
\int_{\R^N} |A(x,u_n)| ||\nabla u_n|^{p-2}\nabla u_n - |\nabla u|^{p-2}\nabla u||\nabla v|dx\ \to \ 0 
\quad\hbox{uniformly with respect to $v$.}
\]
Moreover,  H\"older inequality and \eqref{space1} imply that
\[
\int_{\R^N} |A(x,u_n) - A(x,u)| |\nabla u|^{p-1}|\nabla v|dx\le
\left(\int_{\R^N} |A(x,u_n) - A(x,u)|^{\frac{p}{p-1}} |\nabla u|^{p}dx\right)^{\frac{p-1}{p}},
\]
where
\[
\int_{\R^N} |A(x,u_n) - A(x,u)|^{\frac{p}{p-1}} |\nabla u|^{p}dx \ \to \ 0 \quad\hbox{uniformly with respect to $v$}
\]
from the Lebesgue Dominated Convergence Theorem, as $(H_0)$ and \eqref{succ3} imply
\[
|A(x,u_n) - A(x,u)|^{\frac{p}{p-1}} |\nabla u|^{p} \to 0 \qquad \hbox{a.e in $\R^N$}
\]
and from \eqref{bn1} it follows 
\[
|A(x,u_n) - A(x,u)|^{\frac{p}{p-1}} |\nabla u|^{p} \le 2 c_1^{\frac{p}{p-1}} |\nabla u|^p \quad \hbox{a.e. in $\R^N$, with}
\ |\nabla u|^p \in L^1(\R^N).
\]
Finally, by using again \eqref{space1} we have
\[\begin{split}
&\int_{\R^N} |A_t(x,u_n)|\ | |\nabla u_n|^p - |\nabla u|^p|\ |v|dx + \int_{\R^N} |A_t(x,u_n) - A_t(x,u)|\ |\nabla u|^p|v|dx\\
&\qquad 
\le \int_{\R^N} |A_t(x,u_n)|\ | |\nabla u_n|^p - |\nabla u|^p|dx
 + \int_{\R^N} |A_t(x,u_n) - A_t(x,u)|\ |\nabla u|^p dx,
\end{split}
\]
where, by reasoning as in the first part of this proof 
but replacing $A(x,t)$ with $A_t(x,t)$, we obtain
\[
\int_{\R^N} |A_t(x,u_n)|\ | |\nabla u_n|^p - |\nabla u|^p|dx \to 0,\quad
\int_{\R^N} |A_t(x,u_n) - A_t(x,u)| |\nabla u|^p dx \to 0.
\]
Hence, summing up, 
$|\langle d\J_1(u_n) - d\J_1(u),v\rangle| \to 0$ uniformly with respect to $v$
if $\|v\|_X =1$, i.e., \eqref{bnn1} is completely proved.\\
At last, we claim that
\begin{equation}\label{bnn2}
\J_2(u_n) \to \J_2(u)\quad \hbox{and}\quad \|d\J_2(u_n) - d\J_2(u)\|_{X'} \to 0
\quad\hbox{if $\ n\to+\infty$.}
\end{equation}
To this aim, firstly we note that
\[
|\J_2(u_n) - \J_2(u)|\ \le\ \frac1p\ ||u_n|_p^p - |u|_p^p|\ + \int_{\R^N} |G(x,u_n) - G(x,u)| dx,
\]
where \eqref{Sob1} and \eqref{succ1} imply that
\[
||u_n|_p^p - |u|_p^p| \ \to\ 0.
\]
Furthermore, from \eqref{succ3} and Remark \ref{remG} it follows that
\[
G(x,u_n)\ \to\ G(x,u) \quad \hbox{a.e. in $\R^N$}
\]
and 
\[
|G(x,u_n) - G(x,u)|\ \le\ \frac{a_1}{p} |u_n|^p + \frac{a_2}{q} |u_n|^q 
+ \frac{a_1}{p} |u|^p + \frac{a_2}{q} |u|^q \qquad \hbox{a.e. in $\R^N$,}
\]
where $u_n \to u$ both in $L^p(\R^N)$ and in $L^q(\R^N)$ from Lemma \ref{immergo2}. Hence, $h \in L^1(\R^N)$ exists such that
\[
\frac{a_1}{p} |u_n|^p + \frac{a_2}{q} |u_n|^q 
+ \frac{a_1}{p} |u|^p + \frac{a_2}{q} |u|^q\ \le\ h \qquad \hbox{a.e. in $\R^N$,}
\]
so the Lebesgue Dominated Convergence Theorem implies
\[
\int_{\R^N} |G(x,u_n) - G(x,u)| dx \ \to\ 0.
\]
On the other hand, fixing $v \in X$ such that $\|v\|_X = 1$, 
from H\"older inequality, \eqref{Sob1} and \eqref{space1} we have that
\[
\begin{split}
&|\langle d\J_2(u_n) - \J_2(u),v\rangle|\ \le\  \int_{\R^N} ||u_n|^{p-2}u_n - |u|^{p-2}u|\ |v| \ dx 
+ \int_{\R^N} |g(x,u_n) - g(x,u)|\ |v|\ dx\\
&\qquad\le\  \left(\int_{\R^N} ||u_n|^{p-2}u_n - |u|^{p-2}u|^{\frac{p}{p-1}} dx\right)^{\frac{p-1}{p}}
+ \left(\int_{\R^N} |g(x,u_n) - g(x,u)|^{\frac{p}{p-1}} dx\right)^{\frac{p-1}{p}}.
\end{split}
\]
By reasoning as in the proof of \eqref{stima2}, respectively \eqref{stima3},
but replacing $\nabla u_n$ with $u_n$ and $\nabla u$ with $u$,
from \eqref{succ1} we obtain 
\[
\int_{\R^N} ||u_n|^{p-2}u_n - |u|^{p-2}u|^{\frac{p}{p-1}} dx \ \to \ 0.
\]
Moreover, $(G_0)$ and \eqref{succ3} imply
\[
|g(x,u_n) - g(x,u)|^{\frac{p}{p-1}}\ \to\ 0 \qquad \hbox{a.e. in $\R^N$,}
\]
while $(G_1)$ and direct computations give
\[
|g(x,u_n) - g(x,u)|^{\frac{p}{p-1}}\ \le\ (4 a_1^p)^{\frac{1}{p-1}} (|u_n|^p + |u|^p) + 
(4 a_2^p)^{\frac{1}{p-1}} \big(|u_n|^{p\frac{q-1}{p-1}} + |u|^{p\frac{q-1}{p-1}}\big)
\qquad \hbox{a.e. in $\R^N$,}
\]
where $q \ge p$ and Lemma \ref{immergo2} imply that $h_1 \in L^1(\R^N)$ exists such that
\[
|g(x,u_n) - g(x,u)|^{\frac{p}{p-1}}\ \le\ h_1(x) \qquad \hbox{a.e. in $\R^N$.}
\]
Whence, by applying again the Lebesgue Dominated Convergence Theorem we have that
\[
\int_{\R^N} |g(x,u_n) - g(x,u)|^{\frac{p}{p-1}} dx \ \to\ 0.
\]
Thus, summing up, 
$|\langle d\J_2(u_n) - d\J_2(u),v\rangle| \to 0$ uniformly with respect to $v$
if $\|v\|_X =1$, i.e., \eqref{bnn2} is satisfied.
\end{proof}


\section{Statement of the main result} \label{secmain}

From now on, we assume that in addition to hypotheses $(H_0)$--$(H_1)$ and $(G_0)$--$(G_1)$,
functions $A(x,t)$ and $g(x,t)$ satisfy the following further conditions:
\begin{itemize}
\item[$(H_2)$] a constant $\alpha_0 > 0$ exists such that
\[
A(x,t) \ge \alpha_0 \qquad \hbox{a.e. in $\R^N$, for all $t \in \R$;}
\]
\item[$(H_3)$] some constants $R \ge 1$ and $\alpha_1 > 0$ exist such that
\[
p A(x,t) + A_t(x,t) t \ge \alpha_1 A(x,t)\quad\hbox{a.e. in $\R^N$ if $|t| \ge R$;}
\]
\item[$(H_4)$] some constants $\mu > p$ and $\alpha_2 > 0$ exist such that
\[
(\mu - p) A(x,t) - A_t(x,t) t \ge \alpha_2 A(x,t)\quad\hbox{a.e. in $\R^N$, for all $t \in \R$;}
\]
\item[$(H_5)$] $\ A(x,t) = A(|x|,t)\ $ a.e. in $\R^N$, for all $t \in \R$; 
\item[$(G_2)$] $\; \displaystyle \lim_{t \to 0} \frac{g(x,t)}{|t|^{p-1}}\ =\ 0\;$ uniformly for a.e. $x \in \R^n$;
\item[$(G_3)$] taking $\mu$ as in $(H_4)$, then
\[
0 < \mu G(x,t) \le g(x,t) t\quad\hbox{a.e. in $\R^N$, for all $t \in \R\setminus\{0\}$;}
\]
\item[$(G_4)$] $\ g(x,t) = g(|x|,t)\ $ a.e. in $\R^N$, for all $t \in \R$.
\end{itemize}

\begin{remark}\label{part1}
If we consider the special coefficient 
\begin{equation}
\label{part2}
A(x,t)\ =\ A_1(x) + A_2(x) |t|^{p s} 
\end{equation}
then $(H_0)$--$(H_1)$ are satisfied if $p s > 1$ and $A_1$, $A_2 \in L^\infty(\R^N)$, 
$(H_2)$ follows from 
\[
A_1(x) \ge \alpha_0,\qquad A_2(x) \ge 0\qquad \hbox{a.e. in $\R^N$}
\]
for a constant $\alpha_0 > 0$, $(H_3)$ is always true, $(H_4)$
holds if $\mu > p(1+s)$, while $(H_5)$ reduces to assume that 
both $A_1(x)$ and $A_2(x)$ are radially symmetric. 
\end{remark}

Now, we are able to state our main existence result.

\begin{theorem} \label{mainthm}
Assume that $A(x,t)$ and $g(x,t)$ satisfy
conditions $(H_0)$--$(H_5)$, $(G_0)$--$(G_4)$ with $1 < p < q < p^*$. 
Then, problem \eqref{euler} admits at least one weak bounded radial solution.
\end{theorem}

\begin{remark}
If $A(x,t)$ is as in \eqref{part2} and $g(x,t) = |t|^{\mu-2}t$,
since $(G_0)$--$(G_4)$ hold with $q = \mu > 1$, in the hypotheses pointed out in Remark \ref{part1} 
we have that Theorem \ref{mainthm} applies if 
\begin{equation}
\label{e0p}
1 < p < p+1 < p(1 + s) < \mu < p^*.
\end{equation}
We note that, if $1 < p < N$, from \eqref{e0p}
it has to be $N < p^2+p$. In particular, if $p=2$, we obtain 
a solution of \eqref{euler0} when $3 \le N <6$. Thus, 
Theorem \ref{mainthm} reduces to Theorem \ref{mod0} if $p=2$ and $N=3$.
\end{remark}

In order to prove Theorem \ref{mainthm}, we need some direct consequences 
of the previous assumptions.

\begin{remark}
In $(H_3)$ and $(H_4)$ we can always assume $\alpha_1 < p$ and
$\alpha_2 < \mu-p$. Hence, $\alpha_1 + \alpha_2 < \mu$, and we have
\[
- (p-\alpha_1) A(x,t)\ \le\ A_t(x,t)t\ \le\ (\mu-p-\alpha_2) A(x,t)
\quad\hbox{a.e. in $\R^N$ if $|t| \ge R$.}
\]
Whence, there results
\begin{equation}
\label{cabb}
|A_t(x,t)t|\ \le\ \alpha_3 A(x,t)\quad\hbox{a.e. in $\R^N$ if $|t| \ge R$,}
\end{equation}
with $\alpha_3 = \max\{p-\alpha_1,\mu-p-\alpha_2\}$, which implies
\begin{equation}
\label{cabb1}
|A_t(x,t)|\ \le\ \frac{\alpha_3}R A(x,t)\quad\hbox{a.e. in $\R^N$ if $|t| \ge R$.}
\end{equation}
Moreover, from $(H_0)$--$(H_2)$, $(H_4)$ and direct computations 
it follows that
\begin{equation}
\label{alto}
A(x,t)\ \le\ \alpha_4 + \alpha_5\ |t|^{\mu-p-\alpha_2}\qquad\hbox{a.e. in $\R^N$, for all $t \in\R$,}
\end{equation}
for suitable $\alpha_4$, $\alpha_5 > 0$.
\end{remark}

\begin{remark} \label{bassoG}
Conditions $(G_0)$--$(G_1)$ and $(G_3)$ imply that for any $\eps > 0$ a function $\eta_\eps \in L^\infty(\R^N)$,
$\eta_\eps(x) > 0$ a.e. in $\R^N$, exists such that
\begin{equation}
\label{basso3}
G(x,t)\ \ge\ \eta_\eps(x)\ |t|^\mu\qquad\hbox{a.e. in $\R^N$ if $|t| \ge \eps$.}
\end{equation}
Hence, from \eqref{alto3} and \eqref{basso3} it follows that
\[
p < \mu \le q.
\]
\end{remark}

\begin{remark} 
From assumptions $(G_1)$--$(G_2)$ and direct computations it follows that 
for any $\eps > 0$ a constant $a_\eps > 0$ exists such that
\[
|g(x,t)|\ \le\ \eps |t|^{p-1} + a_\eps |t|^{q-1}\qquad\hbox{a.e. in $\R^N$ for all $t \in \R^N$,}
\]
and then
\begin{equation}
\label{altoGeps}
|G(x,t)|\ \le\ \frac{\eps}{p}\ |t|^{p} + \frac{a_\eps}{q} |t|^{q}\qquad\hbox{a.e. in $\R^N$ for all $t \in \R^N$.}
\end{equation}
\end{remark}

From Proposition \ref{smooth1} it follows that 
looking for weak (bounded) solutions of \eqref{euler} is equivalent to finding
critical points of the $C^1$ functional $\J$, defined as in \eqref{funct},
on the Banach space $X$ introduced in \eqref{space}.

Unluckily, differently from the bounded case, the 
embeddings of $X$ in suitable Lebesgue spaces are only continuous. 
So, in order to overcome the lack of
compactness, we can reduce to work in the space of radial functions 
which is a natural constraint if the problem is radially invariant (see \cite{Pa}).
Thus, in our setting, we consider the space
\begin{equation}\label{spacer}
X_r := W_r^{1,p}(\R^N) \cap L^\infty(\R^N)
\end{equation}
endowed with norm $\|\cdot\|_X$, which has dual space $(X'_r,\|\cdot\|_{X'_r})$.

The following results hold.

\begin{lemma}[Radial Lemma]\label{radiallemma}
If $p > 1$, then every radial function $u\in W_r^{1,p}(\R^N)$ is almost everywhere
equal to a function $U(x)$, continuous for $x \ne 0$, such that
\begin{equation}
|U(x)|\ \leq\ C \frac{\|u\|_{W}}{|x|^{\vartheta}} \quad \hbox{for all $x \in \R^N$ with $|x| \geq 1$,}
\label{lemmaradiale}
\end{equation}
for suitable constants $C$, $\vartheta > 0$ depending only on $N$ and $p$.
\end{lemma}

\begin{proof}
Firstly, we note that by classical results every radial function $u\in W_r^{1,p}(\R^N)$ can be assumed 
to be continuous at all points except the origin (see \cite{Br}).\\
Now, if $p \ge 2$ and $\vartheta= \frac{N -2}{p}$, 
the proof of \eqref{lemmaradiale} follows from  \cite[Lemma A.III]{BL} 
but reasoning as in \cite[Lemma 3.5]{Pi} (see also \cite[Lemma 3.1]{CS2011}).\\
On the other hand, if $1 < p < N$, we prove that 
\eqref{lemmaradiale} holds with $\vartheta= \frac{N - p}{p}$ following the ideas in \cite[Lemma 3.1.2]{BS}.
In fact, by using a density argument, it is enough to prove the inequality for any function 
$u \in C^\infty_0(\R^N) \cap W_r^{1,p}(\R^N)$. Then, $\phi \in C^\infty_0([0,+\infty[)$
exists such that $u(x) = \phi(|x|)$.
If $x \ne 0$, H\"older inequality and direct computations imply that 
\[\begin{split}
|u(x)|\ &=\ |\phi(|x|)| \le \int_{|x|}^{+\infty}|\phi'(r)|dr \le 
\left(\int_{|x|}^{+\infty}|\phi'(r)|^p r^{N-1}dr\right)^{\frac1p}
\left(\int_{|x|}^{+\infty} r^{\frac{1-N}{p-1}}dr\right)^{\frac{p-1}p}\\
&\le\ \omega_{N-1}^{-\frac1p} \left(\omega_{N-1} \int_{0}^{+\infty}|\phi'(r)|^p r^{N-1}dr\right)^{\frac1p}
\left(\frac{|x|^{\frac{p-N}{p-1}}}{\frac{N-p}{p-1}}\right)^{\frac{p-1}p}
\ =\ C  \frac{|\nabla u|_p}{|x|^{\frac{N-p}{p}}},
\end{split}
\]
where $\omega_{N-1}$ is the Lebesgue measure of $\partial B_1(0)$ in $\R^{N-1}$.
\end{proof}

\begin{lemma}\label{immergo3}
If $p > 1$ then the following compact embeddings hold:
\[
 W_r^{1,p}(\R^N)  \hookrightarrow\hookrightarrow L^l(\R^N) \qquad\hbox{for any $p < l < p^*$.}
\]
\end{lemma}

\begin{proof}
The proof is contained essentially in \cite[Theorem 3.2]{CS2011}, anyway, for the sake of completeness,
we give here more details. \\
Taking  $p < l < p^*$, let $(u_{n})_{n}$ be a bounded sequence in $W_r^{1,p}(\R^N)$.
By reasoning as in \cite[Theorem A.I']{BL},  
from \eqref{lemmaradiale} it follows that
$|u_{n}(x)| \to 0$ as $|x| \to +\infty$ uniformly with respect to $n$,
moreover, up to a subsequence, $(u_{n})_{n}$ converges
for a.e. $x \in \R^N$ and weakly in $W_{r}^{1,p}(\R^N)$ to a radial function $u$. 
At last, fixing $l < \tilde{l} < p^*$ and taking
$P(\tau)= |\tau|^{l}$, $Q(\tau)=|\tau|^{p}+|\tau|^{\tilde{l}}$, 
by \cite[Theorem A.I]{BL} we conclude that $u_{n} \to u$ strongly in $L^{l}(\R^N)$.
\end{proof}

\begin{remark} \label{smooth2}
Due to the assumptions $(H_5)$ and $(G_4)$, we can reduce to look for critical points of the restriction of $\J$
in \eqref{funct} to $X_r$, which we still denote as $\J$ for simplicity (see \cite{Pa}).\\
We recall that Proposition \ref{smooth1} 
implies that functional $\J$ is $C^1$ on the Banach space $X_r$, too, if also 
$(H_0)$--$(H_1)$, $(G_0)$--$(G_1)$ hold. 
\end{remark}


\section{Proof of the main result} \label{secproof}

The goal of this section is to apply Theorem \ref{mountainpass} to the functional $\J$ on $X_r$.

For simplicity, in the following proofs, when a sequence $(u_n)_n$ is involved,
we use the notation $(\eps_n)_n$
for any infinitesimal sequence depending only on $(u_n)_n$ while $(\eps_{k,n})_n$
for any infinitesimal sequence depending not only on $(u_n)_n$ but also on some fixed
integer $k$. Moreover, $c_i$ denotes any strictly positive constant independent of $n$.

In order to prove the weak Cerami--Palais--Smale condition,
we need some preliminary lemmas.

Firstly, let us point out that, while if $p > N$ the two norms $\|\cdot\|_X$ and $\|\cdot\|_W$
are equivalent, if $p \le N$ sufficient conditions are required for the boundedness of
a $W^{1,p}$--function on bounded sets as in the following result.

\begin{lemma}\label{tecnico} 
Let $\Omega$ be an open bounded domain in $\R^N$ with boundary $\partial\Omega$, 
consider $p$, $q$ so that $1 < p \le q < p^*$, $p \le N$, and take $v \in W^{1,p}(\Omega)$.
If $\gamma >0$ and $k_0\in \N$ exist such that
\[
 k_0 \ge \esssup_{\partial\Omega} v(x)
\]
and
\[
\int_{\Omega^+_k}|\nabla v|^p dx \le \gamma\left(k^q\ \meas(\Omega^+_k) +
\int_{\Omega^+_k} |v|^q dx\right)\qquad\hbox{for all $k \ge k_0$,}
\]
with $\Omega^+_k = \{x \in \Omega: v(x) > k\}$, then $\displaystyle \esssup_{\Omega} v$
is bounded from above by a positive constant which can be chosen so
that it depends only on $\meas(\Omega)$, $N$, $p$, $q$, $\gamma$, $k_0$,
$|v|_{p^*}$ ($|v|_l$ for some $l > q$ if $p^* = +\infty$). 
Vice versa, if 
\[
- k_0 \le \essinf_{\partial\Omega} v(x)
\]
and inequality
\[
\int_{\Omega^-_k}|\nabla v|^p dx \le \gamma\left(k^q\ \meas(\Omega^-_k) +
\int_{\Omega^-_k} |v|^q dx\right)\qquad\hbox{for all $k \ge k_0$,} 
\]
holds with $\Omega^-_k = \{x \in \Omega: v(x) < - k\}$, then $\displaystyle \esssup_{\Omega}(-v)$ 
is bounded from above by a positive constant which can be
chosen so that it depends only on $\meas(\Omega)$, $N$, $p$, $q$,
$\gamma$, $k_0$, $|v|_{p^*}$ ($|v|_l$ for some $l > q$ if $p^* = +\infty$).
\end{lemma}

\begin{proof}
The proof follows from \cite[Theorem  II.5.1]{LU} but reasoning as in \cite[Lemma 4.5]{CP2}.
\end{proof}

A consequence of Lemma \ref{tecnico} is that the weak limit in $W^{1,p}_r(\R^N)$ of
a $(CPS)_\beta$--sequence has to be bounded in $\R^N$.

\begin{proposition}\label{tecnico3}
Let $1 < p < q < p^*$ and assume that $(H_0)$--$(H_5)$, $(G_0)$--$(G_1)$, $(G_3)$--$(G_4)$ hold.
Then, taking any $\beta \in \R$ and a $(CPS)_\beta$--sequence $(u_n)_n \subset X_r$, 
it follows that $(u_n)_n$ is bounded in $W_r^{1,p}(\R^N)$
and a constant $\beta_0 > 0$ exists such that 
\begin{equation}\label{c5bis}
|u_n(x)| \le \beta_0 \quad \mbox{for a.e. $x \in \R^N$ such that $|x| \ge 1$ and for all $\ n\in \N$.}
\end{equation}
Moreover, there exists $u \in X_r$ such that, up to subsequences, 
\begin{eqnarray}
&&u_n \rightharpoonup u\ \hbox{weakly in $W^{1,p}_r(\R^N)$,}
\label{c2}\\
&&u_n \to u\ \hbox{strongly in $L^l(\R^N)$ for each $l \in ]p,p^*[$,}
\label{c3}\\
&&u_n \to u\ \hbox{a.e. in $\R^N$,}
\label{c4}
\end{eqnarray}
if $n\to+\infty$.
\end{proposition}

\begin{proof}
Let $\beta \in \R$ be fixed and consider a sequence $(u_n)_n \subset X_r$
such that
\begin{equation}\label{c1}
\J(u_n) \to \beta \quad \hbox{and}\quad \|d\J(u_n)\|_{X_r'}(1 + \|u_n\|_X) \to 0\qquad
\mbox{if $\ n\to+\infty$.}
\end{equation}
From \eqref{funct}, \eqref{diff}, \eqref{c1}, $(H_4)$, $(G_3)$ and $(H_2)$ 
we have that
\[
\begin{split}
\mu\beta + \eps_n\ =\ &\mu \J(u_n) - \langle d\J(u_n),u_n\rangle =
\frac{1}p\ 
\int_{\R^N} \big( (\mu - p) A(x,u_n) - A_t(x,u_n) u_n\big) |\nabla u_n|^p dx \\
& + \ \frac{\mu - p}{p} \ \int_{\R^N} |u_n|^p dx + \int_{\R^N} (g(x,u_n) u_n - \mu G(x,u_n)) dx\\
\ge\ &\frac{\alpha_2}p\ \int_{\R^N}  A(x,u_n) |\nabla u_n|^p dx +\ \frac{\mu - p}{p}\ \int_{\R^N} |u_n|^p dx\\
\ge\ &\frac{\alpha_0\alpha_2}p\ \int_{\R^N} |\nabla u_n|^p dx +\ \frac{\mu - p}{p}\ \int_{\R^N} |u_n|^p dx\\
\ge\ &c_1\ \|u_n\|_W^p 
\end{split}
\]
with $c_1 = \frac1p \min\{\alpha_0\alpha_2, \mu - p\}$. Hence,  
$(u_n)_n$ is bounded in $W_r^{1,p}(\R^N)$ and Lemma \ref{radiallemma}
implies the uniform estimate \eqref{c5bis}. Furthermore, $u \in W_r^{1,p}(\R^N)$
exists such that \eqref{c2}--\eqref{c4} hold, up to subsequences. \\
Now, we have just to prove that $u \in L^\infty(\R^N)$.\\
As $u \in W_r^{1,p}(\R^N)$, from  Lemma \ref{radiallemma} it follows that
\begin{equation}
\label{ess0}
\beta_1\ =\ \esssup_{|x| \ge 1} |u(x)| \ < \ +\infty.
\end{equation}
Then, it is enough to prove that 
\begin{equation}
\label{ess00}
\esssup_{|x| \le 1} |u(x)| \ < \ +\infty.
\end{equation}
Arguing by contradiction, let us assume that either
\begin{equation}
\label{ess1}
\esssup_{|x| \le 1} u(x) \ = \ +\infty
\end{equation}
or
\begin{equation}
\label{ess2}
\esssup_{|x| \le 1} (- u(x)) \ = \ +\infty.
\end{equation}
If, for example, \eqref{ess1} holds 
then, for any fixed $k \in\N$,
$k > \max\{\beta_0,\beta_1,R\}$ ($R \ge 1$ as in $(H_3)$, $\beta_0$ as in \eqref{c5bis} and 
$\beta_1$ as in \eqref{ess0}), we have that
\begin{equation}\label{asp}
\meas(B^+_k) > 0 \quad\hbox{with}\quad B^+_k = \{x \in B_1(0): u(x) > k\}.
\end{equation}
We note that the choice of $k$ and \eqref{ess0} imply that
\begin{equation}\label{asp1}
B^+_k \ =\ \{x \in \R^N: u(x) > k\}.
\end{equation}
Moreover, if we set 
\[
 B^+_{k,n}\ =\ \{x \in B_1(0): u_n(x) > k\},\quad n \in \N,
\]
the choice of $k$ and \eqref{c5bis} imply that
\begin{equation}\label{asp2}
B^+_{k,n} \ =\ \{x \in \R^N: u_n(x) > k\} \qquad \hbox{for all $n \in \N$.}
\end{equation}
Now, consider the new function
$R^+_k : t \in\R \mapsto R^+_kt \in \R$ such that
\[
R^+_kt = \left\{\begin{array}{ll}
0&\hbox{if $t \le k$}\\
t-k&\hbox{if $t > k$}
\end{array}\right. .
\]
By definition and \eqref{asp1}, respectively \eqref{asp2}, it results
\begin{equation}\label{asp3}
R^+_k u(x)\ =\ \left\{\begin{array}{ll}
0&\hbox{if $x \not\in B_k^+$}\\
u(x) - k&\hbox{if $x \in B_k^+$}
\end{array}\right. ,\quad
R^+_k u_n(x)\ =\ \left\{\begin{array}{ll}
0&\hbox{if $x \not\in B_{k,n}^+$}\\
u_n(x) - k&\hbox{if $x \in B_{k,n}^+$}
\end{array}\right. ;
\end{equation}
hence, 
\begin{equation}\label{asp4}
R^+_k u \in W_0^{1,p}(B_1(0))\qquad \hbox{and}\qquad R^+_k u_n \in W_0^{1,p}(B_1(0)) \quad \hbox{for all $n \in N$.}
\end{equation}
From \eqref{c2} it follows that $R_k^+u_n \rightharpoonup R_k^+u$
weakly in $W^{1,p}_r(\R^N)$, then, from \eqref{asp4}, in $W_0^{1,p}(B_1(0))$.
As $W_0^{1,p}(B_1(0)) \hookrightarrow\hookrightarrow L^l(B_1(0))$ for any $1 \le l< p^*$, then 
\begin{equation}\label{b01}
\lim_{n\to+\infty} \int_{B_1(0)} |R_k^+u_n|^l dx\ =\  \int_{B_1(0)} |R_k^+u|^l dx \quad \hbox{for any $1 \le l< p^*$.}
\end{equation}
Moreover, from \eqref{c3} we have $u_n \to u$ strongly in $L^l(B_1(0))$ for any $l \in ]p,p^*[$
and then 
\begin{equation}\label{b011}
\lim_{n\to+\infty} \int_{B_1(0)} |u_n|^l dx\ =\  \int_{B_1(0)} |u|^l dx \quad \hbox{for any $1 \le l< p^*$.}
\end{equation}
Thus, by the sequentially weakly lower semicontinuity of the norm $\|\cdot\|_{W}$,
 we have that
\[
\int_{\R^N} |\nabla R_k^+u|^p dx + \int_{\R^N} |R_k^+u|^p dx\ \le\ 
\liminf_{n\to+\infty}\left(\int_{\R^N} |\nabla R_k^+u_n|^p dx + \int_{\R^N} |R_k^+u_n|^p dx\right),
\]
i.e., from \eqref{asp3} -- \eqref{b01} we have
\[
\begin{split}
\int_{B^+_k} |\nabla u|^p dx + \int_{B_1(0)} |R_k^+u|^p dx\ &\le\ 
\liminf_{n\to+\infty}\left(\int_{B^+_{k,n}} |\nabla u_n|^p dx + \int_{B_1(0)} |R_k^+u_n|^p dx\right)\\
&=\ \liminf_{n\to+\infty} \int_{B^+_{k,n}} |\nabla u_n|^p dx + \int_{B_1(0)} |R_k^+u|^p dx.
\end{split}
\]
Hence,
\begin{equation}\label{b1}
\int_{B^+_k} |\nabla u|^p dx \ \le\ \liminf_{n\to+\infty} \int_{B^+_{k,n}} |\nabla u_n|^p dx .
\end{equation}
On the other hand, from $\|R^+_ku_n\|_X \le \|u_n\|_X$ it follows that
\[
|\langle d\J(u_n),R^+_ku_n\rangle| \le \|d\J(u_n)\|_{X'_r}\|u_n\|_X,
\]
then \eqref{c1} and \eqref{asp} imply that $n_{k}\in \N$ exists so that
\begin{equation}\label{b2}
\langle d\J(u_n),R^+_ku_n\rangle < \meas(B^+_k) \qquad \hbox{for all $n \ge n_{k}$.}
\end{equation}
Let us point out that, being $\alpha_1 < p$ and $k >R$, assumption $(H_3)$ implies that
\[
\begin{split}
&\langle d\J(u_n),R^+_ku_n\rangle\ =\ 
\int_{B^+_{k,n}} (1 - \frac k{u_n}) \left(A(x,u_n) + \frac1p A_t(x,u_n)u_n\right) |\nabla u_n|^p dx \\
&\qquad +\ \int_{B^+_{k,n}} \frac k{u_n} A(x,u_n) |\nabla u_n|^p dx 
+ \int_{\R^N} |u_n|^{p-2}u_n R_k^+u_n dx
 - \int_{\R^N} g(x,u_n)R^+_ku_n dx\\
&\quad \ge\ \frac{\alpha_1}{p} \int_{B^+_{k,n}} A(x,u_n) |\nabla u_n|^p dx 
- \int_{\R^N} h(x,u_n) R^+_ku_n dx,
\end{split}
\]
with 
\begin{equation}\label{acca}
h(x,t)\ =\ g(x,t) - |t|^{p-2}t \quad \hbox{for a.e. $x \in \R^N$ and all $t \in \R$.}
\end{equation}
Hence, from $(H_2)$ and \eqref{asp4} it follows that
\begin{equation}\label{b3}
\begin{split}
\frac{\alpha_0 \alpha_1}{p} \int_{B^+_{k,n}} |\nabla u_n|^p dx\ &\le\
\frac{\alpha_1}{p} \int_{B^+_{k,n}} A(x,u_n) |\nabla u_n|^p dx \\
& \le\ \langle d\J(u_n),R^+_ku_n\rangle + \int_{B_1(0)} h(x,u_n) R^+_ku_n dx.
\end{split}
\end{equation}
As $(G_1)$ implies that
\begin{equation}\label{suh}
|h(x,t)|\ \le\ c_2 + c_3 |t|^{q-1} \quad \hbox{for a.e. $x \in B_1(0)$ and all $t \in \R$}
 \end{equation}
for suitable constants $c_2$, $c_3 > 0$, then from \eqref{suh} and \eqref{b01}, \eqref{b011} it follows that
\begin{equation}\label{lim1}
\lim_{n\to+\infty}\int_{B_1(0)} h(x,u_n) R^+_ku_n dx \ =\ \int_{B_1(0)} h(x,u) R^+_ku\ dx.
\end{equation}
Thus, from \eqref{b1}, \eqref{b2}, \eqref{b3} and \eqref{lim1} we obtain that
\begin{equation}\label{b4}
\frac{\alpha_0 \alpha_1}{p} \int_{B^+_{k}} |\nabla u|^p dx\ \le\
\meas(B^+_k) + \int_{B_1(0)} h(x,u) R^+_ku\ dx.
\end{equation}
From \eqref{asp3}, \eqref{suh}, \eqref{b4} and direct computations some positive constants
$c_4$, $c_5 > 0$ exist such that
\[
\int_{B^+_{k}} |\nabla u|^p dx\ \le\
c_4 \meas(B^+_k) + c_5 \int_{B^+_k} |u|^q dx.
\]
As this inequality holds for all $k > \max\{\beta_0,\beta_1,R\}$, 
Lemma \ref{tecnico} implies that \eqref{ess1} is not true. 
Thus, \eqref{ess2} must
hold. In this case, fixing any $k \in\N$, $k > \max\{\beta_0,\beta_1,R\}$, we have
\[
\meas(B^-_k) > 0,\qquad \hbox{with $B^-_k= \{x \in B_1(0): u(x) < - k\}$,}
\]
and we can consider
$R^-_k : t \in\R \mapsto R^-_kt \in \R$ such that
\[
R^-_kt = \left\{\begin{array}{ll}
0&\hbox{if $t \ge -k$}\\
t+k&\hbox{if $t <- k$}
\end{array}\right. .
\]
Thus, reasoning as above, but replacing $R^+_k$ with $R^-_k$, and again by means
of Lemma \ref{tecnico} we prove  that \eqref{ess2} cannot hold.
Hence, \eqref{ess00} has to be true.
\end{proof}

At last, by using some ideas contained in the proof of \cite[Theorem A.1]{BL}, we are able 
to state the following compactness result.

\begin{lemma}\label{tecnico2} 
Assume that $g(x,t)$ satisfies conditions $(G_0)$--$(G_2)$ and $(G_4)$, 
with  $1 < p \le q < p^*$, and consider $(w_n)_n$, $(v_n)_n \subset X_r$ and $w \in X_r$ such that
\begin{equation}\label{suglim1}
\|w_n\|_W \le M_1 \quad \hbox{for all $n \in \N$,} \qquad
w_n \to w \quad\hbox{a.e. in $\R^N$,}
\end{equation}
and
\begin{equation}\label{suglim2}
\|v_n\|_X \le M_2 \quad \hbox{for all $n \in \N$,}\qquad 
v_n \to 0 \quad\hbox{a.e. in $\R^N$,}
\end{equation}
for some constants $M_1$, $M_2 > 0$.
Then,
\begin{equation}\label{suglim}
\lim_{n\to+\infty} \int_{\R^N} g(x,w_n) v_n dx\ =\ 0.
\end{equation}
\end{lemma}

\begin{proof}
Firstly, we note that from \eqref{suglim1}, \eqref{suglim2} and $(G_0)$ it follows that
\begin{equation}\label{qo}
g(x,w_n) v_n \to 0\quad \hbox{a.e. in $\R^N$.}
\end{equation}
Now, fixing any bounded Borel set $B$, we prove that 
\begin{equation}\label{suglim3}
\lim_{n\to+\infty} \int_{B} g(x,w_n) v_n dx\ =\ 0.
\end{equation}
In fact, from $(G_1)$ and applying twice Young inequality once with $\frac{p}{p-1}$ and 
its conjugate exponent $p$ and once with $\frac{p^*}{q-1}$
and its conjugate exponent $\frac{p^*}{p^* - q+1}$ if $p <N$ (otherwise,
if $p=N$ it is enough to replace $p^*$ with any $l > q$), the 
estimate in \eqref{suglim2} implies
\begin{equation}\label{suglim5}
\begin{split}
|g(x,w_n)v_n|\ &\le\ (a_1 |w_n|^{p-1} + a_2 |w_n|^{q-1}) |v_n|_\infty \le c_1 + c_2 |w_n|^{p} + c_3 |w_n|^{p^*}\\
&\le\ c_4 + c_5 |w_n|^{p^*}
\qquad\hbox{a.e. in $\R^N$,}
\end{split}
\end{equation}
for suitable constants $c_i > 0$. 
Then, considering the function $\varphi(k) = \left(\frac{k-c_4}{c_5}\right)^{\frac{1}{p^*}}$
if $k > c_4$, we have that
\begin{equation}\label{suglim7}
\varphi(k) \to +\infty \quad \hbox{if $k\to+\infty$}.
\end{equation}
From \eqref{suglim5} it results
\[
C_{k,n} \subset D_{k,n} \quad \hbox{with}\quad
C_{k,n} = \{x \in \R^N: |g(x,w_n) v_n| \ge k\}, \quad
D_{k,n} = \{x \in \R^N: |w_n| \ge \varphi(k)\},
\]
thus,
\begin{equation}\label{suglim6}
\int_{B\cap C_{k,n}} |g(x,w_n) v_n| dx\ \le\ \int_{B\cap D_{k,n}} |g(x,w_n) v_n| dx\qquad \hbox{for all $k > c_4$, $n \in \N$.}
\end{equation}
We note that, again, from $(G_1)$ it follows that
\[
\frac{|g(x,t)|}{|t|^{p^*}} \ \le\ \frac{a_1 |t|^{p-1} + a_2 |t|^{q-1}}{|t|^{p^*}}\ \to\ 0\quad
\hbox{a.e. in $\R^N$ if $|t| \to +\infty$.}
\] 
Therefore, from \eqref{suglim2} and \eqref{suglim7}, fixing any $\eps > 0$ a constant $k_\eps > 0$
exists such that 
\[
|g(x,w_n) v_n|\ \le \ \eps M_2 |w_n|^{p^*} \quad \hbox{for all $x \in D_{k,n}$, $k \ge k_\eps$, $n \in \N$;}
\]
hence, \eqref{Sob1}, \eqref{suglim1} and \eqref{suglim6} imply that
\[
\int_{B\cap C_{k,n}} |g(x,w_n) v_n| dx\ \le\ \eps c_7\qquad \hbox{for all $k > k_\eps$, $n \in \N$,}
\]
where $c_7 > 0$ is independent of $n$, $\eps$ and $k$.\\
So, the sequence of functions $(g(x,w_n) v_n)_n$ is equi--integrable on $B$ (see, e.g., \cite[Definition 21.1]{Bau}); 
thus, \eqref{suglim3} follows from \eqref{qo} and \cite[Theorems 20.5 and 21.4]{Bau}.\\
On the other hand, from $(G_2)$, fixing any $\eps > 0$ a constant $\delta_\eps > 0$ exists such that
\[
|g(x,t)| < \eps |t|^{p-1}\quad \hbox{a.e. in $\R^N$ if $|t| < \delta_\eps$,}
\]
then, from Lemma \ref{radiallemma} and \eqref{suglim1} a radius $R_\eps >0$ exists 
such that
\[
|w_n(x)| < \delta_\eps\quad \hbox{if $|x| > R_\eps$, for any $n \in \N$;}
\]
whence,
\[
|g(x,w_n) v_n|\ \le\ \eps |w_n|^{p-1} |v_n| \quad \hbox{if $|x| > R_\eps$, for any $n \in \N$.}
\]
Thus, from H\"older inequality and \eqref{suglim1}, \eqref{suglim2} it follows that 
\[
\int_{\R^N\setminus B_{R_\eps}(0)} |g(x,w_n) v_n| dx\ \le\ \eps c_8\qquad \hbox{for all $n \in \N$,}
\]
with $c_8 = M_1^{p-1}M_2 > 0$; so \eqref{suglim} follows from this last estimate and \eqref{suglim3} with $B = B_{R_\eps}(0)$. 
\end{proof}

Now, we are ready to prove the $(wCPS)$ condition in $\R$ by adapting the arguments developed 
in \cite[Proposition 3.4]{CP1}, or also \cite[Proposition 4.6]{CP2}, 
to our setting in the whole space $\R^N$.

\begin{proposition}\label{wCPS}
If $1 < p < q < p^*$ and $(H_0)$--$(H_5)$, $(G_0)$--$(G_4)$ hold, then 
functional $\J$ satisfies the weak Cerami--Palais--Smale condition in $X_r$
at each level $\beta \in \R$.
\end{proposition}

\begin{proof}
Let $\beta \in \R$ be fixed and consider a sequence $(u_n)_n \subset X_r$
such that \eqref{c1} holds.
By applying Proposition \ref{tecnico3} the uniform estimate \eqref{c5bis} holds 
and there exists $u \in X_r$ such that,
up to subsequences, \eqref{c2}--\eqref{c4} are satisfied.\\
For simplicity, the last part of the proof is divided in the following three steps:
\begin{itemize}
\item[1.] defining $T_k : \R \to \R$ such that
\begin{equation}\label{troncok}
T_kt = \left\{\begin{array}{ll}
t&\hbox{if $|t| \le k$}\\
k\frac t{|t|}&\hbox{if $|t| > k$}
\end{array}\right. ,
\end{equation}
with $k \ge \max\{|u|_\infty, R, \beta_0\} + 1$ ($R \ge 1$ as in $(H_3)$ and $\beta_0$ as in \eqref{c5bis}),
then, as $n \to +\infty$, we have
\begin{equation}\label{pp}
\J(T_ku_n) \to \beta 
\end{equation}
and 
\begin{equation}\label{p2}
\|d\J(T_ku_n)\|_{X_r'} \to 0;
\end{equation}
\item[2.] $\|u_n - u\|_{W} \to 0$ if $n\to+\infty$, as
\begin{equation}\label{eq4}
\|T_ku_n - u\|_{W} \to 0 \qquad \hbox{as $n \to +\infty$;}
\end{equation}
\item[3.] $\J(u) = \beta$ and $d\J(u) = 0$.
\end{itemize}
\smallskip

\noindent
{\sl Step 1.} 
Taking any $k > \max\{|u|_\infty, R, \beta_0\}$,
if we set 
\begin{equation}\label{asp511}
 B_{k,n}\ =\ \{x \in B_1(0): |u_n(x)| > k\},\quad n \in \N,
\end{equation}
the choice of $k$ and \eqref{c5bis} imply that
\begin{equation}\label{asp5}
B_{k,n} \ =\ \{x \in \R^N: |u_n(x)| > k\} \qquad \hbox{for all $n \in \N$.}
\end{equation}
Then, from \eqref{troncok} and \eqref{asp5} we have that
\begin{equation}\label{asp51}
T_k u_n(x)\ =\ \left\{\begin{array}{ll}
u_n(x)&\hbox{for a.e. $x \not\in B_{k,n}$}\\
k\frac{u_n(x)}{|u_n(x)|}&\hbox{if $x \in B_{k,n}$}
\end{array}\right. 
\end{equation}
and
\[
|T_ku_n|_\infty \le k, \quad \|T_ku_n\|_{W} \le \|u_n\|_{W}\qquad \hbox{for each $n \in \N$.}
\]
Now, we define $R_k : \R \to \R$ such that
\[
R_kt = t - T_kt = \left\{\begin{array}{ll}
0&\hbox{if $|t| \le k$}\\
t-k\frac t{|t|}&\hbox{if $|t| > k$}
\end{array}\right..
\]
From \eqref{asp51} it results
\begin{equation}\label{asp51bis}
R_k u_n(x)\ =\ \left\{\begin{array}{ll}
0&\hbox{for a.e. $x \not\in B_{k,n}$}\\
u_n(x) - k\frac{u_n(x)}{|u_n(x)|}&\hbox{if $x \in B_{k,n}$}
\end{array}\right. ;
\end{equation}
hence, \eqref{asp511} implies that
\begin{equation}\label{asp7}
R_k u_n \in W_0^{1,p}(B_1(0)) \qquad \hbox{for all $n \in N$.}
\end{equation}
Since $k > |u|_\infty$, we have that 
\[
T_ku(x) = u(x) \quad \hbox{and}\quad R_ku(x) = 0 \qquad \hbox{for a.e. $x \in \R^N$;} 
\]
thus, from \eqref{c2} it follows that $R_ku_n \rightharpoonup 0$
weakly in $W^{1,p}_r(\R^N)$, and, from \eqref{asp7}, in $W_0^{1,p}(B_1(0))$.
Since $W_0^{1,p}(B_1(0)) \hookrightarrow\hookrightarrow L^l(B_1(0))$ for any $1 \le l< p^*$, then 
from \eqref{asp511} and \eqref{asp51bis} we have that
\begin{equation}\label{b012}
\lim_{n\to+\infty} \int_{\R^N} |R_ku_n|^l dx\ =\  0 \quad \hbox{for any $1 \le l< p^*$.}
\end{equation}
On the other hand, reasoning as in the proof of \eqref{b3} but replacing $R^+_ku_n$
with $R_ku_n$ we obtain 
\begin{equation}\label{b8}
\begin{split}
\frac{\alpha_0 \alpha_1}{p} \int_{B_{k,n}} |\nabla u_n|^p dx\ &\le\
\frac{\alpha_1}{p} \int_{B_{k,n}} A(x,u_n) |\nabla u_n|^p dx \\
& \le\ \langle d\J(u_n),R_ku_n\rangle + \int_{B_1(0)} h(x,u_n) R_ku_n dx.
\end{split}
\end{equation}
We note that, from \eqref{suh}, H\"older inequality, \eqref{c3}  and \eqref{b012} it follows that
\begin{equation}\label{lim12}
\lim_{n\to+\infty}\int_{B_1(0)} h(x,u_n) R_ku_n dx \ =\ 0,
\end{equation}
while \eqref{c1} implies that
\begin{equation}\label{lim13}
\lim_{n\to+\infty}|\langle d\J(u_n),R_ku_n\rangle|\ =\ 0
\end{equation}
as $\|R_ku_n\|_X \le \|u_n\|_X$. 
Thus, summing up, from \eqref{b8}--\eqref{lim13} we obtain that
\begin{equation}\label{b111}
 \lim_{n\to+\infty} \int_{B_{k,n}} |\nabla u_n|^p dx\ =\ 0
\end{equation}
and 
\begin{equation}\label{b113}
\lim_{n\to+\infty} \int_{B_{k,n}} A(x,u_n) |\nabla u_n|^p dx \ =\ 0.
\end{equation}
Hence, from \eqref{b012} and \eqref{b111} it follows that
\begin{equation}\label{b112}
 \lim_{n\to+\infty} \|R_k u_n\|_W\ =\ 0.
\end{equation}
Moreover, from \eqref{c4}, \eqref{asp511} and $k > |u|_\infty$ we obtain
\begin{equation}
\label{lim6}
\lim_{n\to+\infty} \meas(B_{k,n})\ =\ 0,
\end{equation}
which, together with \eqref{b011}, implies 
\begin{equation}
\label{lim7}
\lim_{n\to+\infty} \int_{B_{k,n}} |u_n|^l dx\ =\ 0\qquad \hbox{for any $1 \le l< p^*$.}
\end{equation}
Now, we prove that \eqref{pp} holds. 
In fact, from \eqref{funct} and \eqref{asp51} we have
\begin{equation}\label{vai}
\begin{split}
\J(T_ku_n)\  =\ &\frac1p \int_{\R^N \setminus B_{k,n}} A(x,u_n) |\nabla u_n|^p dx
 + \frac1p \int_{\R^N \setminus B_{k,n}} |u_n|^p dx
+ \frac{k^p}p \meas(B_{k,n})\\
& - \int_{\R^N} G(x,T_ku_n) dx\\
=\ & \J(u_n) - \frac1p \int_{B_{k,n}} A(x,u_n) |\nabla u_n|^p dx
 - \frac1p \int_{B_{k,n}} |u_n|^p dx
+ \frac{k^p}p \meas(B_{k,n})\\
& - \int_{\R^N} (G(x,T_ku_n) - G(x,u_n)) dx.
\end{split}
\end{equation}
Thus, \eqref{alto3}, \eqref{asp51}, \eqref{lim6} and \eqref{lim7} imply that
\begin{equation}
\label{lim8}
\lim_{n\to+\infty} \int_{\R^N} (G(x,T_ku_n) - G(x,u_n)) dx\ =\ 0.
\end{equation}
Then, \eqref{pp} follows from \eqref{c1}, \eqref{b113}, \eqref{lim6} -- \eqref{lim8}. \\
In order to prove \eqref{p2}, we take $v \in X_r$ such that $\|v\|_X = 1$; hence,
$|v|_\infty \le 1$, $\|v\|_{W} \le 1$.
From \eqref{diff} and \eqref{asp51} we have that
\[
\begin{split}
&\langle d\J(T_ku_n),v\rangle\  =\  \int_{\R^N \setminus B_{k,n}} A(x,u_n) |\nabla u_n|^{p-2}\nabla u_n \cdot \nabla v dx
 + \frac1p \int_{\R^N \setminus B_{k,n}} A_t(x,u_n) v |\nabla u_n|^{p}dx\\
&\qquad + \int_{\R^N \setminus B_{k,n}} |u_n|^{p-2}u_n v dx + k^{p-1} \int_{B_{k,n}} \frac{u_n}{|u_n|} v dx
- \int_{\R^N} g(x,T_ku_n) v dx\\
&\quad =\  \langle d\J(u_n),v\rangle - \int_{B_{k,n}} A(x,u_n) |\nabla u_n|^{p-2}\nabla u_n \cdot \nabla v dx
 - \frac1p \int_{B_{k,n}} A_t(x,u_n) v |\nabla u_n|^{p}dx\\
&\qquad - \int_{B_{k,n}} |u_n|^{p-2}u_n v dx + k^{p-1} \int_{B_{k,n}} \frac{u_n}{|u_n|} v dx
+ \int_{B_{k,n}} (g(x,u_n) - g(x,T_ku_n)) v dx,
\end{split}
\]
where from \eqref{c1} it follows
\[
|\langle d\J(u_n),v\rangle| \le  \|d\J(u_n)\|_{X_r'}\ \to\ 0,
\]
\eqref{cabb1} and \eqref{b113} imply
\[
\left|\int_{B_{k,n}} A_t(x,u_n) v |\nabla u_n|^{p}dx \right|\ \le\
\frac{\alpha_3}R\int_{B_{k,n}}A(x,u_n)|\nabla u_n|^p dx\ \to\ 0,
\]
from \eqref{lim7} and H\"older inequality we have
\[
\left|\int_{B_{k,n}} |u_n|^{p-2}u_n v dx\right|\ \le\
 \left(\int_{B_{k,n}} |u_n|^{p} dx\right)^{\frac{p-1}{p}}\ \to\ 0,
\]
\eqref{lim6} implies 
\[
\left|\int_{B_{k,n}} \frac{u_n}{|u_n|} v dx\right| \le \meas(B_{k,n}) \ \to\ 0
\]
and from $(G_1)$, \eqref{asp51}, \eqref{lim6}, \eqref{lim7}
and H\"older inequality it follows that
\[
\begin{split}
\left|\int_{B_{k,n}} (g(x,u_n) - g(x,T_ku_n)) v dx\right|\ \le\ &a_1 
\left(\int_{B_{k,n}} |u_n|^{p} dx\right)^{\frac{p-1}{p}} 
+ a_2 \left(\int_{B_{k,n}} |u_n|^{q} dx\right)^{\frac{q-1}{q}}\\
&+ (a_1 k^{p-1} + a_2 k^{q-1})\ \meas(B_{k,n}) \quad \to\ 0.
\end{split}
\]
Thus, summing up, we obtain 
\begin{equation}\label{eq1}
|\langle d\J(T_ku_n),v\rangle| \le \eps_{k,n} + \left|\int_{B_{k,n}} A(x,u_n) |\nabla u_n|^{p-2}\nabla u_n \cdot \nabla v dx\right|.
\end{equation}
Now, in order to estimate the last integral in \eqref{eq1}, following 
the notations introduced in the proof of Proposition \ref{tecnico3}, let us
consider the set $B^+_{k,n}$ and the test function defined as
\[
\varphi^+_{k,n} = v R^+_{k}u_n.
\]
By definition, we have
\[
\|\varphi^+_{k,n}\|_X \le 2 \|u_n\|_X,
\]
and, thus, \eqref{c1} implies 
\[
\|d\J(u_n)\|_{X_r'}\|\varphi^+_{k,n}\|_X\ \le\ \eps_n.
\]
From definition \eqref{acca}, \eqref{asp3} and direct computations we note that
\[
\begin{split}
\langle d\J(u_n),\varphi^+_{k,n}\rangle\ =\ & \int_{B^+_{k,n}} A(x,u_n) R^+_{k}u_n |\nabla u_n|^{p-2}
\nabla u_n \cdot \nabla v dx +  \int_{B^+_{k,n}} A(x,u_n)v|\nabla u_n|^p dx\\
&+ \frac1p \int_{B^+_{k,n}}A_t(x,u_n)v R^+_{k}u_n|\nabla u_n|^p dx - \int_{B^+_{k,n}} h(x,u_n)v R^+_{k}u_n dx,
\end{split}
\]
where, since $B_{k,n}^+ \subset B_{k,n}$, from \eqref{lim6} it follows that
\[
\lim_{n\to+\infty} \meas(B_{k,n}^+)\ =\ 0,
\]
while \eqref{b113}, \eqref{cabb}, \eqref{suh}, \eqref{lim7} and direct computations imply
\[
\left|\int_{B^+_{k,n}} A(x,u_n)v|\nabla u_n|^p dx\right|\ \le\  \int_{B_{k,n}^+} A(x,u_n) |\nabla u_n|^p dx \ \to\ 0,
\]
\[
\left|\int_{B^+_{k,n}}A_t(x,u_n)v R^+_{k}u_n|\nabla u_n|^p dx\right|\ \le\  
\alpha_3 \int_{B_{k,n}^+} A(x,u_n) |\nabla u_n|^p dx \ \to\ 0,
\]
\[
\left|\int_{B^+_{k,n}} h(x,u_n)v R^+_{k}u_n dx\right| \le \int_{B_{k,n}^+} |h(x,u_n)|u_n dx 
\le c_2\int_{B_{k,n}^+} u_n dx + c_3\int_{B_{k,n}^+} |u_n|^q dx\ \to\ 0.
\]
Hence, summing up, we have
\begin{equation}\label{bbis}
\left|\int_{B^+_{k,n}} A(x,u_n) R^+_{k}u_n |\nabla u_n|^{p-2}
\nabla u_n \cdot \nabla v dx\right|\ \le\  \eps_{k,n}.
\end{equation}
Now, if we fix $k > \max\{|u|_\infty, R, \beta_0\} + 1$, all the previous computations hold also 
for $k-1$ and then in particular, \eqref{b111} and \eqref{bbis} become
\begin{equation}\label{b111bis}
 \lim_{n\to+\infty} \int_{B_{k-1,n}} |\nabla u_n|^p dx\ =\ 0
\end{equation}
and
\begin{equation}\label{bbis1}
\left|\int_{B^+_{k-1,n}} A(x,u_n) R^+_{k-1}u_n |\nabla u_n|^{p-2}
\nabla u_n \cdot \nabla v dx\right|\ \le\  \eps_{k,n}.
\end{equation}
Since $B^+_{k,n} \subset B^+_{k-1,n}$, then 
\[
\begin{split}
&\int_{B^+_{k-1,n}} A(x,u_n) R^+_{k-1}u_n |\nabla u_n|^{p-2}
\nabla u_n \cdot \nabla v dx\ =\ \int_{B^+_{k,n}} A(x,u_n) R^+_{k-1}u_n |\nabla u_n|^{p-2}
\nabla u_n \cdot \nabla v dx\\
&\qquad \qquad  + 
\int_{B^+_{k-1,n} \setminus B^+_{k,n}} A(x,u_n) R^+_{k-1}u_n |\nabla u_n|^{p-2}
\nabla u_n \cdot \nabla v dx\\
&\quad\qquad =\ \int_{B^+_{k,n}} A(x,u_n) R^+_{k}u_n |\nabla u_n|^{p-2}
\nabla u_n \cdot \nabla v dx + \int_{B^+_{k,n}} A(x,u_n) |\nabla u_n|^{p-2}
\nabla u_n \cdot \nabla v dx\\
&\qquad\qquad +
\int_{B^+_{k-1,n} \setminus B^+_{k,n}} A(x,u_n) R^+_{k-1}u_n |\nabla u_n|^{p-2}
\nabla u_n \cdot \nabla v dx,
\end{split}
\]
where $(H_1)$, \eqref{asp3}, the properties of $B^+_{k-1,n} \setminus B^+_{k,n}$, H\"older inequality and 
\eqref{b111bis} imply
\[\begin{split}
\left|\int_{B^+_{k-1,n} \setminus B^+_{k,n}} A(x,u_n) R^+_{k-1}u_n |\nabla u_n|^{p-2}
\nabla u_n \cdot \nabla v dx\right|\ &\le\ c_6 \int_{B^+_{k-1,n} \setminus B^+_{k,n}} |\nabla u_n|^{p-1}
|\nabla v| dx \\
&\le\ c_6 \left(\int_{B^+_{k-1,n}} |\nabla u_n|^{p} dx\right)^{\frac{p-1}p} \ \to\ 0.
\end{split}
\]
Thus, from \eqref{bbis} and \eqref{bbis1} it follows that
\[
\left|\int_{B^+_{k,n}} A(x,u_n) |\nabla u_n|^{p-2}
\nabla u_n \cdot \nabla v dx\right|\ \le\ \eps_{k,n}.
\]
Similar arguments apply also if we consider $B^-_{k,n}$
and the test functions
\[
\varphi^-_{k,n} = v R^-_{k}u_n,\quad \varphi^-_{k-1,n} = v R^-_{k-1}u_n;
\]
hence, we have
\begin{equation}\label{eq3}
\left|\int_{B_{k,n}}  A(x,u_n)|\nabla u_n|^{p-2} \nabla u_n \cdot \nabla v dx\right|\le \eps_{k,n}.
\end{equation}
Thus, \eqref{p2} follows from \eqref{eq1} and \eqref{eq3} as all $\eps_{k,n}$ are independent
of $v$.
\smallskip

\noindent
{\sl Step 2.}
We note that \eqref{c2}--\eqref{c4} imply that, if $n\to+\infty$,
\begin{eqnarray}
&&T_k u_n \rightharpoonup u\ \hbox{weakly in $W^{1,p}_r(\R^N)$,}
\label{cc2}\\
&&T_ku_n \to u\ \hbox{strongly in $L^l(\R^N)$ for each $l \in ]p,p^*[$,}
\label{cc3}\\
&&T_ku_n \to u\ \hbox{a.e. in $\R^N$.}
\label{cc4}
\end{eqnarray}
Now, as in \cite{AB1}, let us consider the real map 
\[
\psi: t \in \R \mapsto \psi(t) = t \e^{\eta t^2}\in \R,
\] 
where $\eta > (\frac\beta{2\alpha})^2$ will be fixed once $\alpha$, $\beta > 0$ are chosen in a suitable way later on. By
definition,
\begin{equation}\label{psi1}
\alpha \psi'(t) - \beta |\psi(t)| > \frac\alpha 2\qquad \hbox{for all $t \in \R$.}
\end{equation}
If we define $v_{k,n} = T_ku_n - u$, from the choice of $k$ we have that $|v_{k,n}|_\infty \le 2k$ for all $n \in \N$,
hence
\begin{equation}\label{psi2}
|\psi(v_{k,n})| \le \psi(2k),\quad 0<\psi'(v_{k,n}) \le \psi'(2k)
\qquad\hbox{a.e. in $\R^N$ for all $n\in \N$,}
\end{equation}
while from \eqref{cc4} it follows that
\begin{equation}\label{psi3}
\psi(v_{k,n}) \to 0, \quad
\psi'(v_{k,n}) \to 1 \qquad\hbox{a.e. in $\R^N$ if $n\to +\infty$.}
\end{equation}
Moreover, we note that
\[
|\psi(v_{k,n})|	\ \le\ |v_{k,n}| \e^{4 k^2 \eta}
\qquad\hbox{a.e. in $\R^N$ for all $n\in \N$,}
\]
thus, direct computations imply that
$(\|\psi(v_{k,n})\|_X)_n$ is bounded, and so from \eqref{psi3},
up to subsequences, it is 
\begin{equation}\label{psi5}
\psi(v_{k,n}) \rightharpoonup 0 \quad \hbox{weakly in $W^{1,p}_r(\R^N)$,}
\end{equation}
while from \eqref{p2} it follows that
\[
\langle d\J(T_ku_n),\psi(v_{k,n})\rangle \to 0 \quad \hbox{as $n\to +\infty$,}
\]
where
\[
\begin{split}
\langle d\J(T_ku_n),\psi(v_{k,n})\rangle\  =\ & \int_{\R^N \setminus B_{k,n}} A(x,u_n) \psi'(v_{k,n})
|\nabla u_n|^{p-2}\nabla u_n \cdot \nabla v_{k,n} dx\\
& + \frac1p \int_{\R^N \setminus B_{k,n}} A_t(x,u_n) \psi(v_{k,n}) |\nabla u_n|^{p}dx
 + \int_{\R^N \setminus B_{k,n}} |u_n|^{p-2}u_n \psi(v_{k,n}) dx\\
& + k^{p-1} \int_{B_{k,n}} \frac{u_n}{|u_n|} \psi(v_{k,n}) dx
- \int_{\R^N} g(x,T_ku_n) \psi(v_{k,n}) dx.
\end{split}
\]
From \eqref{lim6} and \eqref{psi2} we have 
\[
\lim_{n\to+\infty} \int_{B_{k,n}} \frac{u_n}{|u_n|} \psi(v_{k,n}) dx \ =\ 0,
\]
while, from Lemma \ref{tecnico2} with $w_n=T_ku_n$ and $v_n = \psi(v_{k,n})$,
it follows that
\[
\lim_{n\to+\infty} \int_{\R^N} g(x,T_ku_n) \psi(v_{k,n}) dx\ =\ 0.
\]
Hence, summing up, we obtain that 
\[
\begin{split}
\eps_{k,n}\  \ge \ & \int_{\R^N \setminus B_{k,n}} A(x,u_n) \psi'(v_{k,n})
|\nabla u_n|^{p-2}\nabla u_n \cdot \nabla v_{k,n} dx\\
& + \frac1p \int_{\R^N \setminus B_{k,n}} A_t(x,u_n) \psi(v_{k,n}) |\nabla u_n|^{p}dx
 + \int_{\R^N \setminus B_{k,n}} |u_n|^{p-2}u_n \psi(v_{k,n}) dx.
\end{split}
\]
Now, we note that, since $|u_n(x)| \le k$ for all $n\in \N$ a.e. in $\R^N \setminus B_{k,n}$,
from $(H_1)$ and $(H_2)$ a constant $c_7 > 0$, which depends only on $k$, exists such that
\[
\begin{split}
\left|\int_{\R^N \setminus B_{k,n}} A_t(x,u_n) \psi(v_{k,n}) |\nabla u_n|^{p}dx\right|\ &\le\ 
\frac{c_7}{\alpha_0} \int_{\R^N \setminus B_{k,n}} A(x,u_n) |\psi(v_{k,n})| |\nabla u_n|^{p}dx\\
&=\ \frac{c_7}{\alpha_0} \int_{\R^N \setminus B_{k,n}} A(x,u_n) |\psi(v_{k,n})| |\nabla u_n|^{p-2}
\nabla u_n\cdot \nabla v_{k,n} dx \\
&\quad +
\frac{c_7}{\alpha_0} \int_{\R^N \setminus B_{k,n}} A(x,u_n) |\psi(v_{k,n})| |\nabla u_n|^{p-2}
\nabla u_n\cdot \nabla u dx,
\end{split}
\]
where, the boundedness of $(u_n)_n$ in $W^{1,p}_r(\R^N)$, $(H_1)$, H\"older inequality, \eqref{psi2},
\eqref{psi3} and the Lebesgue Dominated Convergence Theorem imply
\[ 
\begin{split}
0\ &\le\ \left|\int_{\R^N \setminus B_{k,n}} A(x,u_n) |\psi(v_{k,n})| |\nabla u_n|^{p-2}\nabla u_n\cdot \nabla u dx\right|\\
&\le\ c_8 \left(\int_{\R^N \setminus B_{k,n}} |\psi(v_{k,n})|^p |\nabla u|^{p} dx\right)^{\frac1p}
\left(\int_{\R^N \setminus B_{k,n}} |\nabla u_n|^{p} dx\right)^{\frac{p-1}{p}}\ \to\ 0.
\end{split}
\]
Thus, setting
\[
h_{k,n}(x) = \psi'(v_{k,n})
-\ \frac{c_7}{p\alpha_0} |\psi(v_{k,n})|,
\]
and choosing, in the definition of function $\psi$, constants $\alpha = 1$ and $\beta = \frac{c_7}{p\alpha_0}$,
from \eqref{psi1} it is
\begin{equation}\label{psi6}
h_{k,n}(x) \ > \ \frac12 \quad \hbox{a.e. in $\R^N$.}
\end{equation}
Moreover, we have that
\[
\begin{split}
\eps_{k,n}\  &\ge \ \int_{\R^N \setminus B_{k,n}} h_{k,n} A(x,u_n) |\nabla u_n|^{p-2}\nabla u_n \cdot \nabla v_{k,n} dx
+ \int_{\R^N \setminus B_{k,n}} |u_n|^{p-2}u_n \psi(v_{k,n}) dx\\
&=\  \int_{\R^N \setminus B_{k,n}} A(x,u) |\nabla u|^{p-2}\nabla u \cdot \nabla v_{k,n} dx\\
&\quad + \int_{\R^N \setminus B_{k,n}} \left(h_{k,n} A(x,u_n) - A(x,u)\right)|\nabla u|^{p-2}\nabla u
 \cdot \nabla v_{k,n} dx\\
&\quad + \int_{\R^N \setminus B_{k,n}} h_{k,n} A(x,u_n) 
(|\nabla u_n|^{p-2}\nabla u_n - |\nabla u|^{p-2}\nabla u)
\cdot \nabla v_{k,n} dx\\
&\quad + \int_{\R^N \setminus B_{k,n}} (|u_n|^{p-2}u_n - |u|^{p-2}u) \psi(v_{k,n}) dx + 
\int_{\R^N \setminus B_{k,n}} |u|^{p-2}u \psi(v_{k,n}) dx,
\end{split}
\]
with
\[
\lim_{n\to+\infty}\int_{\R^N \setminus B_{k,n}} A(x,u) |\nabla u|^{p-2}\nabla u \cdot \nabla v_{k,n} dx = 0,\qquad
\lim_{n\to+\infty} \int_{\R^N \setminus B_{k,n}} |u|^{p-2}u \psi(v_{k,n}) dx = 0
\]
from \eqref{cc2}, respectively \eqref{psi5}.
On the other hand, we note that $(H_0)$, \eqref{c4} and \eqref{psi3} imply that
$h_{k,n} A(x,u_n) - A(x,u) \to 0$ a.e. in $\R^N$; hence, since $(\|v_{k,n}\|_W)_n$ is bounded,
from H\"older inequality and Lebesgue Dominated Convergence Theorem it follows that
\[
\begin{split}
&\left|\int_{\R^N \setminus B_{k,n}} \left(h_{k,n} A(x,u_n) - A(x,u)\right)|\nabla u|^{p-2}\nabla u
 \cdot \nabla v_{k,n} dx\right| \\
&\quad \le\ 
\left(\int_{\R^N \setminus B_{k,n}} |h_{k,n} A(x,u_n) - A(x,u)|^{\frac{p}{p-1}}|\nabla u|^{p}\right)^{\frac1p}
\|v_{k,n}\|_W \ \to \ 0 .
\end{split}
\]
Thus, summing up, for the strong convexity of the power function with exponent $p > 1$, \eqref{psi6}, $(H_2)$ and
$\e^{\eta v_{k,n}^2} \ge 1$ we obtain
\[
\begin{split}
\eps_{k,n}\  \ge \ \frac{\alpha_0}{2} &\int_{\R^N \setminus B_{k,n}}  
(|\nabla u_n|^{p-2}\nabla u_n - |\nabla u|^{p-2}\nabla u) \cdot \nabla v_{k,n} dx\\
& + \int_{\R^N \setminus B_{k,n}} (|u_n|^{p-2}u_n - |u|^{p-2}u) v_{k,n} dx \ge 0, 
\end{split}
\]
which implies
\[
\int_{\R^N \setminus B_{k,n}}  
(|\nabla u_n|^{p-2}\nabla u_n - |\nabla u|^{p-2}\nabla u) \cdot \nabla v_{k,n} dx
+ \int_{\R^N \setminus B_{k,n}} (|u_n|^{p-2}u_n - |u|^{p-2}u) v_{k,n} dx \to 0,
\]
or better, from \eqref{cc2},
\begin{equation}\label{lim10}
\lim_{n\to+\infty}
\left(\int_{\R^N \setminus B_{k,n}}  
|\nabla u_n|^{p-2}\nabla u_n \cdot \nabla v_{k,n} dx
+ \int_{\R^N \setminus B_{k,n}} |u_n|^{p-2}u_n v_{k,n} dx\right) = 0.
\end{equation}
Since from \eqref{lim6} we have that 
\[
\lim_{n\to+\infty} \int_{B_{k,n}} |T_ku_n|^{p-2}T_ku_n\ v_{k,n} dx = 0
\]
and 
\[
\int_{\R^N \setminus B_{k,n}}  
|\nabla u_n|^{p-2}\nabla u_n \cdot \nabla v_{k,n} dx
= \int_{\R^N}  
|\nabla T_k u_n|^{p-2}\nabla T_ku_n \cdot \nabla v_{k,n} dx,
\]
\eqref{lim10} becomes
\begin{equation}\label{lim14}
\lim_{n\to+\infty}
\left(\int_{\R^N}  
|\nabla T_ku_n|^{p-2}\nabla T_ku_n \cdot \nabla v_{k,n} dx
+ \int_{\R^N} |T_ku_n|^{p-2}T_ku_n v_{k,n} dx\right) = 0.
\end{equation}
Now, we define the operator
\[
{\cal L}_p : u \in W^{1,p}_r(\R^N) \mapsto {\cal L}_p u \in (W_r^{1,p}(\R^N))'
\]
as 
\[
\langle{\cal L}_p u,v\rangle \ =\ 
\int_{\R^N} |\nabla u|^{p-2}\nabla u \cdot \nabla v dx 
+ \int_{\R^N} |u|^{p-2}u v dx,\quad v \in W^{1,p}_r(\R^N).
\]
Since $\langle{\cal L}_p u,u\rangle = \|u\|_W^p$ and the H\"older inequalities for integrals
and sums imply that 
\[
|\langle{\cal L}_p u,v\rangle| \le \|u\|_W^{p-1} \|v\|_W,
\]
from \eqref{cc2}, \eqref{lim14} and \cite[Proposition 1.3]{PAO}
it follows that \eqref{eq4} holds.\\
At last, since $\|u_n - u\|_{W} \le \|T_ku_n - u\|_{W}+ \|R_ku_n\|_{W}$,
from \eqref{eq4} and \eqref{b112} it follows that  
$\|u_n - u\|_{W} \to 0$.
\smallskip

\noindent
{\sl Step 3.} As from definition we have $|T_ku_n|_\infty \le k$ for all $n \in \N$, 
the proof follows from \eqref{eq4}, \eqref{cc4}, Proposition \ref{smooth1} and \eqref{pp}, \eqref{p2}. 
\end{proof}

\begin{proof}[Proof of Theorem \ref{mainthm}]
From Remark \ref{smooth2}, looking for weak bounded radial solutions of \eqref{euler}
is equivalent to finding critical points of the $C^1$ functional 
$\J$, defined in \eqref{funct}, restricted to the Banach space $X_r$ as in \eqref{spacer}.
Moreover, by definition, it is $\J(0) = 0$ and Proposition \ref{wCPS} implies
that $\J$ satisfies the weak Cerami--Palais--Smale condition in $\R$.\\
Now, from $(H_2)$, \eqref{altoGeps} with $\eps = \frac{\alpha_0}{2}$, 
\eqref{Sob1} and direct computations it follows that
\[
\J(u)\ \ge\ \frac{\alpha_0}{2p}\ \|u\|_W^p\ -\ \frac{a_\eps}{q}\ \sigma_q^q\ \|u\|_W^q
\qquad \hbox{for any $u \in X_r$;}
\]
hence, $r$, $\varrho > 0$ exist such that
\[
u \in X_r, \; \|u\|_W = r\quad \then\quad J(u) \ge \varrho.
\]
On the other hand, fixing any $\bar{u} \in X_r$ such that
$\meas(C_1) > 0$, with $C_1 = \{x \in \R^N: |\bar{u}(x)| > 1\}$,
for any $t > 1$ from \eqref{alto}, Remark \ref{bassoG} with $\eps=1$, $(G_3)$ and direct computations
we have
\[
\begin{split}
\J(t\bar{u})\ &=\ \frac{t^p}{p}\  \int_{\R^N} A(x,t\bar{u})|\nabla \bar{u}|^p dx + 
\frac{t^p}{p}\ \int_{\R^N} |\bar{u}|^p dx - \int_{\R^N} G(x,t\bar{u}) dx\\
&\le\ c_1 t^p \|\bar{u}\|_X^p+ c_2 t^{\mu - \alpha_2}\|\bar{u}\|_X^{\mu - \alpha_2}
- \int_{C_1} G(x,t\bar{u}) dx - \int_{\R^N\setminus C_1} G(x,t\bar{u}) dx\\
&\le\ c_1 t^p \|\bar{u}\|_X^p+ c_2 t^{\mu - \alpha_2}\|\bar{u}\|_X^{\mu - \alpha_2}
- t^\mu \int_{C_1} \eta_1(x) |\bar{u}|^\mu dx. 
\end{split}
\]
Then, $p < \mu$ implies that $\J(t\bar{u}) \to -\infty$ as $t \to +\infty$; hence,
$e \in X_r$ exists such that $\|e\|_W > r$ and $\J(e) < \varrho$.
So, the conclusion follows from Theorem \ref{mountainpass}.
\end{proof}


\end{document}